\documentclass[letterpaper, 11pt]{article}

\usepackage{amsthm}
\usepackage{amssymb}
\usepackage{amsmath}
\usepackage[margin=1in]{geometry}

\newtheorem{prop}{Proposition}
\newtheorem{lemma}{Lemma}
\newtheorem{thm}{Theorem}
\newtheorem{cor}{Corollary}

\theoremstyle{definition}
\newtheorem{defn}{Definition}

\numberwithin{prop}{subsection}
\numberwithin{lemma}{subsection}
\numberwithin{thm}{subsection}
\numberwithin{cor}{subsection}
\numberwithin{defn}{subsection}

\DeclareMathOperator{\ord}{ord}
\DeclareMathOperator{\Res}{Res}

\DeclareMathOperator{\MinResLoc}{MinResLoc}

\DeclareMathOperator{\Bary}{Bary}

\DeclareMathOperator{\Vol}{Vol}

\DeclareMathOperator{\SL}{SL}
\DeclareMathOperator{\SO}{SO}

\DeclareMathOperator{\SU}{SU}

\DeclareMathOperator{\Min}{Min}
\DeclareMathOperator{\Sphere}{S}
\DeclareMathOperator{\pcap}{pcap}
\DeclareMathOperator{\dH}{\textbf{d}_{\hh}}
\DeclareMathOperator{\dB}{\textbf{d}_{\mathbb{B}}}
\DeclareMathOperator{\dSL}{\textbf{d}_{\textrm{SL}}}

\newcommand{\zetaG}{\zeta_{\text{Gauss}}}

\newcommand{\deldel}[1]{\frac{\partial}{\partial #1}}
\newcommand{\deldeltwo}[1]{\frac{\partial^2}{\partial #1^2}}

\newcommand{\zbar}{\overline{z}}
\newcommand{\fre}{f_{\textrm{Re}}}
\newcommand{\fim}{f_{\textrm{Im}}}

\newcommand{\vn}{\vec{n}}

\newcommand{\vv}{\vec{v}}

\newcommand{\CC}{\mathbb{C}}
\newcommand{\EE}{\mathbb{E}}
\newcommand{\PP}{\mathbb{P}}
\newcommand{\RR}{\mathbb{R}}
\newcommand{\BB}{\mathbb{B}}
\newcommand{\NN}{\mathbb{N}}
\newcommand{\hh}{\mathbb{H}}

\renewcommand{\H}{\textrm{H}}

\newcommand{\del}{\partial}

\title{Hyperbolic Equivariants of Rational Maps}
\date{\today}
\author{Kenneth Jacobs}

\begin{document}

\maketitle

\begin{abstract}
Let $K$ denote either $\RR$ or $\CC$. In this article, we introduce two new equivariants associated to a rational map $f\in K(z)$. These objects naturally live on a real hyperbolic space, and carry information about the action of $f$ on $\PP^1(K)$. When $K=\CC$ we relate the asymptotic behavior of these equivariants to the conformal barycenter of the measure of maximal entropy. We also give a complete description of these objects for rational maps of degree $d=1$. The constructions in this article are based on work of Rumely in the context of rational maps over non-Archimedean fields; similarities between the two theories are highlighted throughout the article.
\end{abstract}

\section{Introduction}

In this paper we introduce two new equivariants and one new invariant associated to the conjugation of a rational map $f\in K(z)$ by $SL_2(K)$, where $K=\RR$ or $K=\CC$. The first equivariant is a function $R_F$ defined on real hyperbolic $[K:\RR]+1$ space that measures the distortion of the unit sphere in $K^2$ induced by homogeneous lifts $F^\gamma$ of $\SL_2(K)$-conjugates $f^\gamma$. The second equivariant is the set of points at which $R_F$ is minimized. The motivation for these objects comes from analogous constructions due to Rumely \cite{Ru1, Ru2} for rational maps defined over complete, algebraically closed, non-Archimedean fields; throughout the article we will explain the connections between ours and Rumely's constructions.

Let $K=\RR$ or $\CC$, and let $f\in K(z)$ have degree $d\geq 1$. Fix a homogeneous lift $F:K^2 \to K^2$ of $f$, given by a pair of coprime polynomials $F = (F_0, F_1)$, $F_i\in K[X,Y]$ of degree $d$ satisfying $f(z) = \frac{F_0(z,1)}{F_1(z,1)}$. Let $S_K$ denote the unit sphere (with respect to the Euclidean norm $||\cdot||$) in $K^2$, and let $d\sigma_K$ be the volume form on $S_K$ normalized to have total volume 1. The quantity
\[
R(F) = \int_{S_K} \log ||Fz|| d\sigma_K
\] is a measurement of distortion of the sphere induced by $F$; this distortion can be expressed explicitly in terms of Alexander's projective capacity (\cite{Al}; see Theorem~\ref{thm:pcap} below).

For $\gamma\in \SL_2(K)$, let $F^\gamma(z) = \gamma^{-1} \cdot F(\gamma\cdot z)$, where $\gamma$ acts on $K^2$ by left multiplication. The matrices that preserve the Euclidean norm on $K^2$ form a group $\SU_2(K)$, and one checks directly that $R(F^\tau) = R(F)$ for $\tau\in \SU_2(K)$. Thus, the function $R_F(\gamma):= R(F^\gamma)$ defined on $\SL_2(K)$ descends to a well defined function
\[
R_F: \SL_2(K) / \SU_2(K) \to \RR\ .
\] We observe that the space $\SL_2(K)/\SU_2(K)$ is naturally isometric to real hyperbolic $[K:\RR]+1$ space, which for the moment we denote $\mathfrak{h}_K$. The following theorem collects the basic properties of this function:
\begin{thm}\label{thm:basicprops}
Let $f\in K(z)$ with degree $d\geq 2$, and fix a homogeneous lift $F$ of $f$. The function $R_F:\mathfrak{h}_K \to \RR$ is smooth, proper, and subharmonic with respect to the hyperbolic Laplacian. In particular, it attains a minimum on $\mathfrak{h}_K$.
\end{thm} 

Theorem~\ref{thm:basicprops} follows from explicit calculations in which we estimate the growth rate of $R_F$ as $[\gamma]$ `approaches the boundary of $\mathfrak{h}_K$' (Theorem~\ref{thm:growthrates}) and provide an explicit expression for the hyperbolic Laplacian of $R_F$ (Theorem~\ref{thm:hypLap}). 

The function $R_F$ is also equivariant, in the sense that $R_F([\gamma]) = R_{F^\gamma}([\textrm{id}])$ where $[\textrm{id}]$ is the (class of the) identity matrix in $\SL_2(K)/\SU_2(K)$. Consequently, the minimal value of $R_F$ is an invariant of the lift $F$, and we define
\begin{defn}
The {\em min invariant} of $f\in K(z)$ is defined to be
\[
m_K(f):= \min_{[\gamma]\in \SL_2(K)/\SU_2(K)} R_F([\gamma]) - \frac{1}{2d} \log |\Res(F_0, F_1)|\ ;
\] here, $\Res(F_0, F_1)$ is the homogeneous resultant of the lift $F$.
\end{defn} The normalization via the resultant addresses the fact that $R_F$ scales logarithmically for different lifts of $F$, i.e. if $c\in K^\times$ then $R_{cF} = R_F + \log |c|$. While there are many ways that one could normalize $R_F$ to address this issue of scaling, our choice is based on analogy with the non-Archimedean setting where the analogous normalization via the resultant has arithmetic significance: in that context, $m_K(f) = 0$ if and only if $f$ has potential good reduction (see \cite{Ru1}).

When $K=\CC$, we will also analyze the set
\[
\Min(f) = \{[\gamma]\in \SL_2(K)/\SU_2(K)\ : \ [\gamma] \textrm{ minimizes } R_F\}\ .
\] For this, the notion of conformal barycenters of measures on $S^2$ will be useful. Douady and Earle introduced the notion of conformal barycenters in \cite{DE} in order to study conformally natural extensions of homeomorphisms of the sphere to the hyperbolic ball. More precisely, the barycenter of an admissible measure $\mu$ on the sphere $S^2$ is a distinguished point of the unit ball in $\RR^3$ whose existence can be established by showing that it is the minimum of a certain convex function $h_\mu$. Details will be given below in Section~\ref{sect:barycentersDE}.

When $K=\CC$, let $\omega_\CC$ denote the Fubini-Study form on $\PP^1(\CC)$, and for $f\in \CC(z)$ let 
\[
\omega_f:= f^* \omega_{\CC} + f_* \omega_{\CC}\ ,
\] where the pushforward is in the sense of measures. Then $\omega_f$ is a positive measure on $\PP^1(\CC)$ of total mass $d+1$; its pullback to $S^2$ under stereographic projection will be denoted $\widehat{\omega_{f}}$. The following theorem gives a geometric interpretation of the conjugates realizing the minimum of $R_F$:
\begin{thm}\label{thm:minimizersintro}
Identifying $\SL_2(\CC)/\SU_2(\CC)$ with the hyperbolic ball $\BB\subseteq \RR^3$, the hyperbolic gradient of $R_F$ is
\[
\nabla_h R_F(\xi) = \int_{S^2} \zeta\ \widehat{\omega_{f^{\gamma_\xi}}}\ ,
\] where $[\gamma_\xi]\in \SL_2(\CC)/\SU_2(\CC)$ is the class corresponding to $\xi\in \BB$. In particular, if $[\gamma]\in \SL_2(\CC) / \SU_2(\CC)$ is in $\Min(f)$, then $\Bary(\widehat{\omega_{f^\gamma}})$ is the origin in the unit ball.
\end{thm} Thus, conjugates for which $R_F$ is minimized correspond to `balanced' representatives of $f$, where balanced is understood in terms of the measure $\omega_f$. In the non-Archimedean setting, classes $[\gamma]$ belonging to the analogue of $\Min(f)$ correspond to conjugates $f^\gamma$ that have semi-stable reduction in the sense of Geometric Invariant Theory (\cite{Ru2}). While there is no natural notion of reduction in the complex setting, one can view `minimizing $R_F$' as a complex analogue of having semi-stable reduction, and Theorem~\ref{thm:minimizersintro} gives a geometric interpretation of what `semi-stable' might mean in this context. It would be very interesting to understand whether the classes $[\gamma]$ attaining the minimum have an interpretation in terms of the moduli space $\mathcal{M}_d$ of degree $d$ rational maps. We remark that the use of barycenters in studying moduli-theoretic questions has already been carried out by DeMarco \cite{De2}, who has shown that the barycenter of the measure of maximal entropy can be used to give a compactification of $\mathcal{M}_2$. 

We next turn our attention to the asymptotic behavior of $R_F$. The following theorem describes the limits of $R_{F^n}$ and $\Min(f^n)$. Recall that $h_\mu$ is a convex function on the unit ball (with the hyperbolic metric) introduced by Douady and Earle that is minimized precisely on the barycenter of measures on the sphere.
\begin{thm}\label{thm:asymptoticsoverC}
Let $f\in \CC(z)$ have degree $d\geq 2$, and let $\mu_f$ be its measure of maximal entropy on $\PP^1(\CC)$. 
\begin{enumerate}
\item The functions $\frac{1}{d^n} R_{F^n}$, viewed as functions on the unit ball $\BB\subseteq \RR^3$, converge locally uniformly to $h_{\mu_f} + C_F$ for an explicit constant $C_F$. 

 \item The sets $\Min(f^n)$, viewed as subsets of the unit ball $\BB\subseteq \RR^3$, converge in the Hausdorff topology to the conformal barycenter $\Bary(\mu_f)$ of the measure of maximal entropy. 
\end{enumerate}
\end{thm}

The investigations carried out in this article were motivated by an analogous investigation of Rumely \cite{Ru1, Ru2} in the context of non-Archimedean dynamics. When $K$ is a complete, algebraically closed, non-Archimedean valued field and $f\in K(z)$, then the analogue of our function $R_F$ is Rumely's $\ord\Res_f$ (there is a canonical, arithmetically-motivated normalization of the lift $F$ in the non-Archimedean setting, see \cite{Ru1}); while we haven't precisely defined $R_F$ in the non-Archimedean context, this can be done, and the resulting function satisfies
\[
\ord\Res_f(\gamma) = -\log|\Res(F_0, F_1)| + 2dR_F(\gamma)\ ,
\] where $\Res(F_0, F_1)$ is the homogeneous resultant of the lift $F$. Thus, our Theorem~\ref{thm:basicprops} is re-captures parts of \cite{Ru1} Theorem 1.1 in the complex setting. The non-Archimedean version of Theorem~\ref{thm:asymptoticsoverC} is contained in \cite{KJ} Theorems 1 and 3. Throughout the paper, we will point out additional similarities between our $R_F$ and Rumely's $\ord\Res_f$. One part of the non-Archimedean picture that is lacking from our constructions is the crucial measure of a rational map (see \cite{Ru2}, Definition 9 and Corollary 6.5). It would be very interesting to develop an analogue of these measures in the Archimedean setting.

We remark that the non-Archimedean and Archimedean perspectives can be combined to give an invariant for rational maps defined over a global field $k$: by summing the min-invariant over all places of $k$ one obtains a function $\hat{h}_R:\mathcal{M}_d \to \RR$. It is natural to ask whether this is comparable to a Weil height. By work of Doyle-Jacobs-Rumely \cite{DJR}, the answer is yes when $d=2$ and $k$ is a function field; similar work of a VIGRE research group at the University of Georgia \cite{HMRVW} gives an affirmative answer if one restricts to polynomials under affine conjugation, again defined over function fields. The general question is currently under investigation.

The article is organized as follows: in Section~\ref{sect:bg} we provide the necessary background information for the remainder of the article. In Section~\ref{sect:bp} we establish basic properties of $R_F$, including explicit growth formulas for $R_F$ at the boundary of $\mathfrak{h}_K$ and an explicit expression for its hyperbolic Laplacian. Section~\ref{sect:asymp} discusses the asymptotic behavior of $R_{F^n}$ and $\Min(f^n)$. Finally, in Section~\ref{sect:done} we compute $\Min(f)$ and $m_K(f)$ explicitly for maps $f$ of degree 1.

\subsection*{Acknowledgements}
The author would like to thank Robert Rumely and Laura DeMarco for fruitful discussions and encouragement in carrying out this project.

\section{Background and Notation}\label{sect:bg}

\subsection{Fundamental Objects}

Let $K$ denote either the field $\RR$ or $\CC$. We endow $K^2$ with the Euclidean norm, $||(X,Y)|| = \sqrt{|X|^2 + |Y|^2}$. Under the usual left action of $\SL_2(K)$ on $K^2$ by left multiplication, the Euclidean norm is preserved by the subgroup
\[
\SU_2(K) = \left\{\begin{matrix} \SO(2)\ , & K = \RR \\[5pt] \SU(2) \ , & K= \CC\end{matrix}\right.\ .
\]
We endow $\PP^1(K)$ with a volume form $\omega_K$ as follows: first let $d\ell_K$ be given locally by
\[
d \ell_K = \left\{\begin{matrix} d\alpha \ , & K= \RR\\[5pt] \frac{i}{2} d\alpha\wedge d\overline{\alpha}\ , & K = \CC\end{matrix}\right.
\] be the top form cooresponding to the usual Lebesgue measure on $K$. Then $\omega_K$ is given locally by
\begin{equation}\label{eq:omegaKdefn}
\omega_K := \left(\frac{1}{1+|\alpha|^2}\right)^{[K:\RR]} \frac{d\ell_K}{\pi} = \left\{\begin{matrix}  \frac{1}{1+\alpha^2} \frac{d\alpha}{\pi} \ , & K= \RR\\[5pt] \frac{i}{2\pi} \frac{1}{(1+|\alpha|^2)^2}\ d\alpha\wedge d\overline{\alpha} \ , & K= \CC\end{matrix}\right.\ .
\end{equation} Note that when $K= \CC$, this is simply the Fubini-Study form. In both cases, $\omega_K$ has been normalized so that $\int_{\PP^1(K)} \omega_K = 1$. Viewing $\SL_2(K)$ as acting on $\PP^1(K)$ by fractional linear transformations, $\omega_K$ is invariant under pullback by elements of $\SU_2(K)$.

The chordal metric on $\PP^1(K)$ is given 
\[
||P,Q|| =\frac{|P-Q|}{\sqrt{1+|P|^2} \cdot \sqrt{1+|Q|^2}}
\] for $P, Q\in K$ and $||P, \infty|| = \frac{1}{\sqrt{1+|P|^2}}$. It is rotation invariant, in the sense that for $\gamma\in \SL_2(K)$ viewed as acting by linear fractional transformations on $\PP^1(K)$, we have $||\gamma(P), \gamma(Q)|| = ||P, Q||$.\\

Throughout this paper we consider rational maps $f\in K(z)$ of degree $d = \textrm{deg } f\geq 1$. A homogeneous lift $F$ of $f$ is a polynomial endomorphism $F: K^2 \to K^2$ given $F = (F_0, F_1)$ for coprime, homogeneous polynomials $F_i \in K[X,Y]$ satisfying $f(z) = \frac{F_0(z,1)}{F_1(z,1)}$. When needed, we will write the $F_i$ with coefficients as
\begin{align*}
F_0(X,Y) &= a_d X^d + ... + a_0 Y^d\\
F_1(X,Y) &= b_d X^d + ... + b_0 Y^d\ 
\end{align*}with $a_i, b_i \in K$. The left multiplication of $\SL_2(K)$ on $K^2$ induces a conjugation of $F$ by
\[
F\mapsto F^\gamma:(X,Y)= \gamma^{-1} \cdot F(\gamma\cdot (X,Y))
\] We write $F^\gamma = (F_0^\gamma, F_1^\gamma)$, where $F_0^\gamma, F_1^\gamma$ are the component polynomials of $F^\gamma$.

We define
\[
R(F) = \int_{S_K} \log ||F(X,Y)|| d\sigma_K\ ,
\] where $S_K \subseteq K^2$ is the unit sphere and $d\sigma_K$ is the unit volume form on $S_K$. This quantity can be interpreted geometrically and will be explored further in Section~\ref{sect:pcap}; of greater interest for our purposes is the function on $\SL_2(K)$ given
\[
R_F(\gamma):= R(F^\gamma)\ .
\] Note that the $\SU_2(K)$-invariance of $\sigma_K$ and $S_K$ implies that $R_F:\SL_2(K)\to \RR$ descends to a well-defined function $R_F:\SL_2(K)/\SU_2(K) \to \RR$. In the next section we will explain how $\SL_2(K)/\SU_2(K)$ is isometric to a real hyperbolic space, and will utilize the geometry of this space to deduce properties of $R_F$. It will also be useful to have an expression for $R_F$ as an integral on $\PP^1(K)$:
%\begin{lemma}\label{lem:integrationalongthefibers}
%For a fixed homogeneous lift $F$ of $f$, we have
\begin{equation}\label{eq:affineequation}
R(F) = \int_{\PP^1(K)} \log \frac{||F(z,1)||}{||(z,1)||^d} \ \omega_K\ .\\
\end{equation}
%\end{lemma}
%\begin{proof}
%Let $\pi:S_K\setminus\{0\}\to \PP^1(K)$ be the projection sending $(X,Y)\in S_K$ to $z = [X:Y]\in \PP^1(K)$. Given $P=[X,Y]\in \PP^1(K)$, the fiber $\pi^{-1}(P)$ is topologically a circle $|\lambda| = 1$ in $K$ (note: this only consists of two points when $K=\RR$). The induced top form $\pi^* \omega_K$ on $\pi^{-1}(P)$ is the uniform measure on this circle. Integration along the fibers gives
%\end{proof}
Let $\hh_K:= \{(z,t)\ : \ z\in K, t>0\}$, and define $P_K(z,t)$ on $\hh_K$ by
\[
P_K(z,t) = \left(\frac{t}{t^2+|z|^2}\right)^{[K:\RR]}\ .
\] This function is the hyperbolic Poisson kernel, used to construct hyperbolic harmonic extensions of functions $g:\PP^1(K) \to \RR$ (see \cite{Sto}, Section 5.6, and also Proposition~\ref{prop:hyperbolicextensions} below). A direct calculation also shows that
\begin{equation}\label{eq:pushforwardomegaK}
(\gamma_{z,t}^{-1})^* \omega_K(\alpha) = P_K(\alpha - z, t) \frac{d\ell_K(\alpha)}{\pi}
\end{equation} where $\gamma_{z,t}(\alpha) = t\alpha + z$. Thus the change of variables formula gives
\[
\int_{\PP^1(K)} Q(\gamma_{z,t}(\alpha)) \ \omega_K(\alpha) = \int_{\PP^1(K)} Q(\alpha)\  (\gamma_{z,t})_* \omega_K(\alpha) = \int_{\PP^1(K)} Q(w) \cdot P_K(z,t; w) d\ell_K(w)\ 
\] for any smooth function $Q$ on $\PP^1(K)$.

%\begin{lemma}\label{lem:pushforwardomegaK}
%Let $\gamma_{z,t}$ be as in (\ref{eq:halfspacerep}), viewed as an affine transformation $\gamma_{z,t}(\alpha) = t\alpha + z$. Then
%\[
%(\gamma_{z,t}^{-1})^* \omega_K(\alpha) = P_K(\alpha-z, t) \frac{d\ell_K(\alpha)}{\pi}
%\]
%\end{lemma}
%\begin{proof}
%Note that $\det D\gamma_{z,t}^{-1} = \frac{1}{t^{[K:\RR]}}$; applying the usual pullback formula to the expression for $\omega_K$ in (\ref{eq:omegaKdefn}) gives
%\[
%(\gamma_{z,t}^{-1})* \omega_K = \left(\frac{1}{1+|\gamma_{z,t}^{-1}(\alpha)|^2}\right)^{[K:\RR]} \frac{1}{t^{[K:\RR]}} d\ell_K
%\] Simplifying this gives the expession in the lemma.
%\end{proof}

\subsection{Models of Hyperbolic Space}\label{sect:spaces}

In the different sections of the paper it will be convenient to work with different models of real hyperbolic 3 space. Let $\mathbb{E}_K = \RR^{[K:\RR]+1}$. We identify three models by specifying a space $X$, a metric $\textrm{d}_X$, and a distinguished point in $X$ which we denote (in all models) by $j$:
\begin{itemize}
\item the upper half space model is the space $(\hh_K, \dH)$ consisting of points $(z,t)$ with $z\in K$ and $t>0$; the distinguished point is $j=(0,1)$.
\item the conformal ball model is the space $(\BB_K, \dB)$, consisting of points $\xi = r \zeta$, where $0\leq r < 1$ and $\zeta$ an element of the unit sphere in $S_K\subseteq \EE_K$; the distinguished point is $j=0$ is the zero vector in $\EE_K$.
\item the quotient space $(\SL_2(K)/\SU_2(K), \dSL)$, consisting of $\SU_2(K)$-equivalence classes of matrices in $\SL_2(K)$; here, the distinguished point is $j=[\textrm{id}]$, the equivalence class of the identity matrix. 
\end{itemize} The metrics $\dH$ and $\dB$ are the usual hyperbolic metrics on the respective model; see, e.g. \cite{Rat}. In particular, if $(x,t), (y,s)\in \hh_K$, then
\begin{equation}\label{eq:hypdist}
\cosh \dH((x,t), (y,s)) = 1 + \frac{ |x-y|^2 + (s-t)^2}{2 s t}\ .
\end{equation} If $x=y = 0$, this reduces\footnote{It is perhaps easier to see this reduced formula using the metric tensor on $\hh_K$ given by $ds_{\hh}^2 = \frac{ds_K^2+dt^2}{t^2}$, where $ds_K^2$ is the Euclidean metric on $K$.} to give $\dH((0,t), (0,s)) =  \left|\log\left(\frac{t}{s}\right)\right|$. The metric $\dSL$ on $\SL_2(K)/\SU_2(K)$ is perhaps less familiar, and will be defined below. 

We begin by recalling some standard decompositions of matrices $\gamma\in \SL_2(K)$. 
\begin{lemma} \label{lem:matrixdecomps}
Let $[\gamma]\in \SL_2(K)/\SU_2(K)$. 
\begin{enumerate}
\item[a)] The class $[\gamma]$ has a unique representative of the form $\gamma = \tau\cdot \eta_A $, where $\tau \in \SU_2(K)$ and $\eta_A = \left(\begin{matrix} e^{A/2} & 0 \\ 0 & e^{-A/2}\end{matrix}\right)$ for some $A\geq 0$.
\item[b)] The class $[\gamma]$ has a unique representative of the form
\begin{equation}\label{eq:halfspacerep}
\gamma_{z,t} = \left(\begin{matrix} \sqrt{t} & \frac{z}{\sqrt{t}} \\ 0 & \frac{1}{\sqrt{t}}\end{matrix}\right)\ ,
\end{equation}where $z\in K$ and $t>0$.
\end{enumerate}
\end{lemma}
\begin{proof}
\begin{itemize}
\item[a)] By \cite{Art}, Theorems 5.4 and 5.8 we can write $\gamma= \tau \eta \tau^*$ where $\tau^*$ is the conjugate transpose of $\tau\in \SU_2(K)$ and $\eta$ is a real diagonal matrix; perhaps replacing $\tau$ by $\tau \begin{pmatrix} 0 & 1 \\ 1 & 0\end{pmatrix}$, we find that $\eta = \eta_A$ for $A\geq 0$. 

\item[b)] This can be checked directly by hand; given a matrix $\gamma'\in \SL_2(K)$, it amounts to solving $\gamma' = \gamma_{z,t}\cdot \tau$ for some $z,t$ and $\tau\in \SU_2(K)$. Uniqueness follows by showing that if $\gamma_{z,t}\cdot \gamma_{z',t'}^{-1}\in \SU_2(K)$, then $z = z'$ and $t=t'$. 
\end{itemize}
\end{proof}

The quantity $A$ appearing in part (a) of Lemma~\ref{lem:matrixdecomps} can be used to define a metric on $\SL_2(K) / \SU_2(K)$:
\begin{defn}
The metric $\dSL$ on $\SL_2(K) / \SU_2(K)$ is defined $\dSL([\gamma], [\omega]) = A$, where $\tau \cdot \eta_A$ is the unique representative of $[\omega^{-1} \cdot \gamma]$ appearing in Lemma~\ref{lem:matrixdecomps} (a). 
\end{defn}

We will show in Proposition~\ref{prop:mapsandmetric} that this is indeed a metric on $\SL_2(K) / \SU_2(K)$. Note that the following invariance property is already clear: for $\alpha, \gamma, \omega\in \SL_2(K)$, we have $\dSL([\alpha\cdot \gamma], [\alpha \cdot \omega]) = \dSL([\gamma], [\omega])$. 

We next recall an isometry between $(\hh_K, \dH)$ and $(\BB_K, \dB)$:
\begin{prop}\label{prop:usualhypisoms}
The spaces $(\hh_K, \dH)$ and $(\BB_K, \dB)$ are isometric. Fix coordinates on the appropriate Euclidean space (depending on $K$) so that $\hh_K = \{(z,t)\ : \ z\in K, t>0\}$ and $\BB_K = \{ r\cdot \zeta\ : \ 0 \leq r < 1\ , \zeta\in S_K\}$. Let $\sigma_K$ denote inversion in the sphere $S_K((0,1), \sqrt{2})$ and let $\eta_K$ denote reflection in the hyperplane $\{(x, 0)\ : x\in K\} \subseteq \EE_K$. Then the map
\[
\iota_K : \BB_K \to \hh_K
\] given $\iota_K = \eta_K\circ \sigma_K$ defines an isometry between $(\BB_K, \dB)$ and $(\hh_K, \dH)$. Moreover, this inversion extends to give stereographic projection $\Sigma: S_K \to \PP^1(K)$ on the boundaries of $\BB_K$ and $\hh_K$. 
\end{prop}
\begin{proof}
See, e.g, \cite{Rat} Sections 4.5 and 4.6.
\end{proof}

There is a natural action of $\SL_2(K)$ on $\hh_K$ that preserves the metric $\dH$, given by the action of fractional linear transformations (\cite{Rat}, Theorem 4.6.2). For our purposes, the action for the matrices given in Lemma~\ref{lem:matrixdecomps} are 
\begin{align*}
\gamma_{z,t} \cdot (0,1) = (z,t)\\
\eta_A \cdot (0,1) = (0, e^A)\ .
\end{align*}

By identifying $\hh_K$ to $\BB_K$ via the map $\iota_K$ appearing in Proposition~\ref{prop:usualhypisoms} we also obtain an action of $\SL_2(K)$ on $\BB_K$. There is also a natural action of $\SL_2(K)$ on $\SL_2(K) / \SU_2(K)$ given by left multiplication. In all cases, we find that $\SU_2(K)$ is the stabilizer subgroup of the distinguished point $j$ in each model. 

We now identify isometries between $\SL_2(K)/ \SU_2(K)$ and the other models of hyperbolic space discussed above:

\begin{prop}\label{prop:mapsandmetric}
The maps
\begin{align*}
\Gamma_{\hh}&: \SL_2(K) / \SU_2(K) \to \hh_K\\
\Gamma_{\BB}&: \SL_2(K)/ \SU_2(K) \to \BB_K
\end{align*}
defined by $\Gamma_{\hh}([\gamma]) = \gamma \cdot (0,1)$ and $\Gamma_{\BB}([\gamma]) = \gamma \cdot 0$ are bijections, and the metric on $\SL_2(K)/ \SU_2(K)$ induced by either of these maps coincides with $\dSL$ introduced above. 
\end{prop}
\begin{proof}
The fact that these maps are bijections follows from the fact that $\SU_2(K)$ is the stabilizer subgroup of the respective distinguished points $j$. To see that the metric induced by these maps agrees with $\dSL$, it's enough to show that the metric induced by $\dH$ agrees with $\dSL$ (since the action of $\SL_2(K)$ on $\BB_K$ was defined by the isometry $\iota_K$ of $\hh_K$ and $\BB_K$). The metric $\dH$ is known to be invariant under the left action of $\SL_2(K)$, and we also saw that $\dSL$ is invariant under the action of $\SL_2(K)$ on $\SL_2(K)/ \SU_2(K)$. Therefore, we only need to check that $\dSL([\gamma], [\textrm{id}]) = \dH(\gamma\cdot(0,1), (0,1))$ for a given $[\gamma]\in \SL_2(K)/ \SU_2(K)$. 

Let $\tau\cdot \eta_A$ be the unique representative of $[\gamma]$ given in Lemma~\ref{lem:matrixdecomps}. Then $\dSL([\gamma], [\textrm{id}]) = A$. We also compute
\begin{align*}
\dH(\gamma \cdot (0,1), (0,1))& = \dH((\tau\cdot \eta_A) \cdot (0,1), (0,1))\\
& = \dH(\eta_A \cdot (0,1), \tau^{-1}\cdot (0,1))\\
& = \dH(\eta_A \cdot (0,1), (0,1))\ .
\end{align*} Above we noted that $\eta_A$ sends $(0,1)$ to $(0, e^{A})$, and from the formula for $\dH$ mentioned above we find that
\[
\dH(\gamma.(0,1), (0,1)) = \dH((0, e^A), (0,1)) = \log(e^A) - \log (1) = A = \dSL([\gamma], [\textrm{id}])\ .
\]
\end{proof}

We close this section with some remarks and conventions that will be used in the rest of the paper:
\begin{itemize}
\item If $f: \SL_2(K)/\SU_2(K) \to \RR$ is a function on $\SL_2(K) / \SU_2(K)$, then we will often write $f(\gamma(j)) = f([\gamma])$ for the corresponding function on $\hh_K$. In a similar way, the function on $\BB_K$ will often be written $f(\gamma(0)) = f([\gamma])$. In all cases, the function name, $f$, will be the same, and we will rely on its argument to determine the domain.

\item Functions and measures defined on $\PP^1(K)$ or $\hh_K$ can be pulled back along stereographic projection to be defined on $S^2$ or $\BB_K$; we will denote these pullbacks by putting a hat $\ \widehat{\cdot}\ $ over the object. For example, if $\mu$ is a probability measure on $\PP^1(K)$ then $\widehat{\mu}= (\iota_K)^* \mu$, where $\iota_K$ is stereographic projection (see Proposition~\ref{prop:usualhypisoms}).
\end{itemize}

\subsection{Potential Theory on $\PP^1(\CC)$}

In the case $K=\CC$, we will make much use of techinques from potential theory. Here, we recall a few facts that we will need.

Viewing points of $\PP^1(\CC)$ in homogeneous coordinates $[x_0 : x_1]$, the patch $U_1 = \{ [x_0:x_1] \ : \ x_1 \neq 0\}$ will be given local coordinates $z = \frac{x_0}{x_1}$ and $\zbar$. The operator $dd^c$ in these coordinates is defined on $\mathcal{C}^2$ functions by
\[
dd^c \psi := \frac{i}{\pi}\psi_{z\overline{z}}\ dz \wedge d\overline{z}\ .
\] We will make use of the following well-known facts pertaining to $dd^c$:

\begin{prop}\label{prop:ddcprops}
Let $\psi:\CC \to \RR$ be $\mathcal{C}^2$. Then
\begin{enumerate}
\item For every $f:\CC\to\CC$ holomorphic, we have
\[
f^* dd^c \psi = dd^c(\psi\circ f)\ .
\]

\item For every $\mathcal{C}^2$ function $\phi:\CC\to \RR$, we have
\[
\int_{\PP^1(\CC)} \phi\ dd^c\psi = \int_{\PP^1(\CC)} \psi\ dd^c \phi\ .
\]
\end{enumerate}
\end{prop}

%\begin{lemma}\label{lem:pullbackformulaforLaplacian}
%Let $\psi:\CC \to \RR$ be $\mathcal{C}^2$, and let $f:\CC\to\CC$ be holomorphic. Then
%\[
%f^* dd^c \psi = dd^c(\psi\circ f)\ .
%\]
%\end{lemma}
%\begin{proof}
%Let $z, \zbar$ be coordinates on $\CC$, so that $\psi$ is a function of $z, \zbar$ and the composition $\psi\circ f$ is $\psi(f, \overline{f})$. The operator $dd^c$ is %given in local coordinates by $dd^c = \frac{i}{2\pi} \partial_{\zbar} \partial_z\ dz \wedge d\overline{z}$; a direct computation with the chain rule gives
%\[
%\partial_z \psi(f, \overline{f}) = \psi_z(f,\overline{f})\cdot f_z + \psi_{\zbar}(f, \overline{f}) \cdot \overline{f}_z = \psi_z(f, \overline{f}) \cdot f_z\ ,
%\] where in the last equality we are using that $\overline{f}_z = 0$ because $f$ is holomorphic. A similar direct computation yields
%\begin{align*}
%\partial_{\zbar} \partial_z (\psi\circ f) &= \left(\psi_{z\zbar}(f, \overline{f}) \cdot f_{\zbar} + \psi_{z\zbar}(f, \overline{f}) \overline{f}_{\zbar}\right) \cdot f_z + \psi_z(f, \overline{f}) f_{z\zbar}\\
%& = \psi_{z\zbar}(f, \overline{f}) |f_z|^2\ .
%\end{align*} Thus in local coordinates on $\PP^1(\CC)$, 
%\[
%dd^c(\psi\circ f) = \frac{i}{2\pi} \psi_{z\zbar}\circ f \cdot |f_z|^2 dz\wedge d\zbar = f^*\left(\frac{i}{2\pi} \psi_{z\zbar}\ dz\wedge d\zbar\right) = f^* dd^c\psi\ .
%\]
%\end{proof}

\subsection{Expressing the Function $R_F$}\label{sect:expressions}
In this section, we collect several for $R_F$ that will be useful throughout the paper. The first computes $R_F([\textrm{id}]) = R(F)$ when $F$ is a linear map:
\begin{lemma}\label{lem:simplestcalculationRF}

If $M\in \SL_2(K)$ with $M = \tau \cdot \eta_A \cdot \sigma$ for $\tau, \sigma\in \SU_2(K)$ and $\eta_A$ as in Lemma~\ref{lem:matrixdecomps}(a), then
\[
\int_{S_K} \log ||M\cdot (X,Y)^\top ||\ d\sigma_K = \left\{\begin{matrix} \log(1+e^A) - \frac{A}{2} - \log 2\ , & K = \RR\\[5pt] -\frac{1}{2} + \frac{A}{2} \left(\frac{e^{2A} +1}{e^{2A}-1}\right)\ , & K = \CC \end{matrix}\right.\ .
\]
\end{lemma}
\begin{proof}
Note that the $\SU_2(K)$- invariance of $S_K$ and $d\sigma_K$ allows us to reduce our calculations to
\[
\int_{S_K} \log ||M\cdot (X,Y)^\top||\ d\sigma_K = \int_{S_K} \log ||\eta_A \cdot (X,Y)^\top||\ d\sigma_K = \frac{1}{2} \int_{S_K} \log \left(e^A |X|^2 + e^{-A}|Y|^2\right)\ d\sigma_K\ .
\] Expressing this as an integral over $\PP^1(K)$ gives
\begin{align*}
\int_{S_K} \log ||M\cdot (X,Y)^\top || d\sigma_K &= \frac{1}{2}\int_{\PP^1(K)} \log \left(\frac{e^A |\alpha|^2 + 1}{|\alpha|^2+1}\right) \omega_K(\alpha) - \frac{A}{2}\\
& = \frac{1}{2} \int_{\PP^1(K)} \log (e^{2A} |\alpha|^2 + 1) \omega_K(\alpha) - \frac{1}{2} \int_{\PP^1(K)}\log (|\alpha|^2 + 1) \omega_K(\alpha) - \frac{A}{2}\ .
\end{align*}

An explicit -- but tedious -- calculation of these integrals in local coordinates gives the expressions asserted in the statement of the lemma.
\end{proof}

We next give a slight refinement of the affine expression for $R_F$ given in (\ref{eq:affineequation}):
\begin{prop}\label{prop:RFLandP}
Let $f\in K(z)$ and $\gamma\in \SL_2(K)$ be represented by $\gamma_{z,t}$ as in Lemma~\ref{lem:matrixdecomps}. Then
\[
R_F([\gamma]) = \frac{1}{2}\int_{\PP^1(K)} \log \frac{|f_1(w)|^2 \cdot \left( |f(w) - z|^2 + t^2\right)}{t^{d+1}} P_K(\alpha-z, t) \frac{d\ell_K}{\pi} +c_K\ ,
\] where $c_K = -\frac{d}{2} \int_{\PP^1(K)} \log (1+|\alpha|^2) \omega_K(\alpha) = \left\{\begin{matrix} -d\log 2\ , & K= \RR\\ -\frac{d}{2}\ , & K=\CC\end{matrix}\right.$. 
\end{prop}
\begin{proof}
Putting the integrand of $R_F$ into affine coordinates yields
\begin{align}
\log \frac{||F^\gamma(X,Y)||}{||X,Y||^d} & = \frac{1}{2} \log \frac{t^{-(d+1)} \left( \left| F_0(t\alpha+z,1) - zF_1(t\alpha+z,1)\right|^2 + t^2 |F_1(t\alpha+z,1)|^2\right)}{ (|\alpha|^2+1)^d}\nonumber\\
& = \frac{1}{2} \log \frac{|f_1(\gamma(\alpha))|^2 \cdot\left( |f({\gamma}(\alpha)) - z|^2 + t^2\right)}{t^{d+1}}-\frac{d}{2} \log (1+|\alpha|^2)\ .\label{eq:simplifiedintegrand}
\end{align} Integrating the second term against $\omega_K$ gives, after a direct calculation, the constant $c_K$. Integrating the first term against $\omega_K$ and using (\ref{eq:pushforwardomegaK}) then gives the asserted expression.
\end{proof}

When $K=\CC$, we can use the complex structure to deduce the following useful expression for $R_F$:
\begin{prop}\label{eq:formulaforgradient}
Let $f\in \CC(z)$, and let $[\gamma]\in \SL_2(\CC)/ \SU_2(\CC)$ be represented by $\gamma_{z,t}$ as in Lemma~\ref{lem:matrixdecomps}. Define $\psi(\alpha; z,t) =-\log||\gamma_{z,t}^{-1}(\alpha), \infty||$ and $\omega_{z,t} = (\gamma_{z,t}^{-1})^* \omega_\CC$. 
\begin{equation}\label{eq:simplifiedRF}
R_F([\gamma])= \frac{d-1}{2} \log t + \int_{\PP^1(\CC)} \psi(\alpha; z,t) f^*  \omega_{z,t} - \frac{d}{2} + \frac{1}{2} \log c(z,t;f)\ ,
\end{equation} where $c(z,t;f) =|a_d - zb_d|^2 + t^2 |b_d|^2$.
\end{prop}
\begin{proof}
Define $\psi_0(z) = \frac{1}{2} \log (1+|z|^2)$; an explicit calculation shows that $dd^c \psi_0 = \omega_\CC - \delta_\infty$ as currents on $\PP^1(\CC)$. We let $[\gamma]\in \SL_2(\CC) / \SU_2(\CC)$ be represented by $\gamma_{z,t}$ as in Lemma~\ref{lem:matrixdecomps}; as a fractional linear transformation on $\CC$, this can be realized as ${\gamma}_{z,t}(\alpha) = t\alpha + z$. Note that
$$\psi(\alpha; z,t)= ({\gamma}_{z,t}^{-1})^* \psi_0(\alpha)\ ,$$ 
so that by Proposition~\ref{prop:ddcprops} we have 
\begin{equation}\label{eq:ddcpsizero}
dd^c \psi(\cdot; z,t) = ({\gamma_{z,t}}^{-1})^* \omega_\CC - \delta_\infty = \omega_{z,t} - \delta_\infty\ .
\end{equation} In local coordinates, $ \omega_{z,t} = \frac{i}{2\pi} \partial_{\overline{\alpha}} \partial_\alpha \psi(\alpha; z,t)\ d\alpha\wedge  d\overline{\alpha}$.

From (\ref{eq:simplifiedintegrand}) we find that 
\begin{align}
\log \frac{||F^\gamma(X,Y)||}{||X,Y||^d} & = \frac{-(d+1)}{2} \log t + \frac{1}{2} \log \frac{ |f_0({\gamma}(\alpha)) - z f_1({\gamma}(\alpha))|^2 + t^2 |f_1({\gamma}(\alpha)|^2}{(|\alpha|^2 + 1)^d}\ ,
\end{align} where $\alpha = \frac{X}{Y}$ is a local coordinate on $\PP^1(\CC)$. As $|\alpha|\to\infty$, the second term tends to $d\log t + \frac{1}{2}\log c(z,t;f)$, where $c(z,t;f)$ is given by
\[
c(z,t;f):= \left\{\begin{matrix} |a_d|^2\ , & f(\infty) = \infty\\ (|f(\infty)-z|^2+t^2) |b_d|^2\ , & f(\infty) \neq \infty\end{matrix}\right. = |a_d - zb_d|^2 + t^2 |b_d|^2\ .
\]
Inserting the expression in (\ref{eq:simplifiedintegrand}) into the definition of $R_F$, we find that

\begin{align}
R_F(\gamma) & = -\frac{d+1}{2} \log t + \frac{1}{2}\int_{\PP^1(\CC)} \log  \frac{ |f_0({\gamma}(\alpha)) - z f_1({\gamma}(\alpha))|^2 + t^2 |f_1({\gamma}(\alpha)|^2}{(|\alpha|^2 + 1)^d} \omega_\CC\nonumber\\
& = -\frac{d+1}{2} \log t + \frac{1}{2} \int_{\PP^1(\CC)}  \log \frac{ |f_0({\gamma}(\alpha)) - z f_1({\gamma}(\alpha))|^2 + t^2 |f_1({\gamma}(\alpha)|^2}{(|\alpha|^2 + 1)^d} (\omega_\CC - \delta_\infty) + \frac{1}{2} \log c(z,t;f) +d\log t\nonumber\\
& = \frac{d-1}{2} \log t + \frac{1}{2} \int_{\PP^1(\CC)} \log  \frac{ |f_0(\alpha) - z f_1(\alpha)|^2 + t^2 |f_1(\alpha)|^2}{(|{\gamma}^{-1}(\alpha)|^2 + 1)^d} ({\gamma}^{-1})^* dd^c \psi_0 + \frac{1}{2} \log c(z,t;f) \nonumber\\
& = \frac{d-1}{2} \log t + \int_{\PP^1(\CC)} \left(f^*\psi(\alpha; z,t) - d \psi(\alpha; z,t) + \log |f_1(\alpha)|\right) dd^c \psi(\cdot; z,t) + \frac{1}{2} \log c(z,t;f)  \ .\label{eq:rewritepotential}
\end{align} We apply Proposition~\ref{prop:ddcprops} to pass the $dd^c$ onto the integrand; note that
\begin{align*}
dd^c f^* \psi(\cdot; z,t) = f^*( \omega_{z,t} - \delta_\infty)&= f^*  \omega_{z,t} - \left(\sum_{f(p_i) = \infty, p_i \neq \infty} \delta_{p_i} + (d-\textrm{deg}(f_1)) \delta_\infty\right)\ ,\\
dd^c \log |f_1(\alpha)| & = \sum_{f(p_i) = \infty,\ p_i \neq \infty} \delta_{p_i} - \textrm{deg}(f_1) \delta_\infty\ .
\end{align*} Inserting these into (\ref{eq:rewritepotential}) and simplifying yields 
\begin{align}
R_F(\gamma) &= \frac{d-1}{2} \log t + \int_{\PP^1(\CC)} \psi(\alpha; z,t) \left(f^*  \omega_{z,t} - d\cdot  \omega_{z,t}\right) + \frac{1}{2} \log c(z,t;f)\nonumber\\
& = \frac{d-1}{2}\log t + \int_{\PP^1(\CC)} \psi(\alpha; z,t) f^*  \omega_{z,t} - d\int_{\PP^1(\CC)} (\gamma^{-1})^* \psi_0(\alpha) (\gamma^{-1})^* \omega_\CC + \frac{1}{2} \log c(z,t;f)\nonumber\\
& = \frac{d-1}{2} \log t + \int_{\PP^1(\CC)} \psi(\alpha; z,t) f^*  \omega_{z,t} - \frac{d}{2} + \frac{1}{2} \log c(z,t;f)\nonumber\ ,
\end{align} which is the asserted formula.
\end{proof}

\subsection{Harmonic Functions on Hyperbolic Space}
The various hyperbolic spaces introduced in Section~\ref{sect:spaces} each come equipped with a Laplace operator; in the geometric setting (i.e. $\hh_K$ and $\BB_K$) this is the Laplace-Beltrami operator that can be computed in terms of the metric tensor for $\dH$ and $\dB$ respectively. On $\SL_2(K) / \SU_2(K)$, the Laplace-Beltrami operator can be interpreted in terms of the Casimir element of the universal enveloping algebra of $\mathfrak{sl}_2(K)$, though we will not make use of this perspective in the present article (see \cite{Ha}). 

Viewing $\hh_K\subseteq \mathbb{E}_K$ with coordinates $(z,t)$ where $z\in K, t>0$, we have (e.g. \cite{Sto}, Exercise 3.5.11)
\[
\Delta_h^{\hh_K} = \left\{\begin{matrix} t^2 \Delta_{\textrm{std}}\ , & K= \RR\\ t^2 \Delta_{\textrm{std}} - t\frac{\partial}{\partial t}\ , & K = \CC\end{matrix}\right.\ ,
\] where $\Delta_{\textrm{std}}$ is the standard Laplace operator on $\EE_K$ given $\Delta_{\textrm{std}} = \sum_{i=1}^{[K:\RR] + 1} \frac{\partial^2}{\partial x_i^2}$.

On $\BB_K$, we will use an expression for the hyperbolic Laplacian in terms of spherical (or polar) coordinates on $\EE_K$ (see, e.g. \cite{Sto}, Exercise 3.5.6)
\[
\Delta_h^{\BB_K} = \frac{1-r^2}{r^2} \left((1-r^2) N^2 +([K:\RR] -1 ) (1+r^2) N + (1-r^2) \Delta_\sigma\right)\ ,
\] where $N = r\frac{\partial}{\partial r}$ and $\Delta_\sigma$ is the part of $\Delta_{\textrm{std}}$ corresponding to the angular coordinates.

A continuous function $g:\PP^1(K) \to \RR$ can be extended to a function $\H\{g\}:\hh_{K} \to \RR$ via the formula
\[
\H\{g\}((z,t)) : = (g * P_K)(z,t) = \int_{K} g(\alpha) P_K(\alpha -z,t) d\ell_K\ ,
\] that is hyperbolic harmonic, i.e $\Delta_h^{\hh_K} H\{g\} = 0$. The function $\H\{g\}$ extends $g$ in the sense that, if $z\in \PP^1(K)$ is viewed as a point in the ideal boundary of $\hh_{K}$, then $\lim_{\hh_{K}\ni w\to z} \H\{g\}(w) = g(z)$ (see \cite{Sto}, Theorem 5.6.2). The next result expresses this extension as a function on $\SL_2(K)/ \SU_2(K)$:

\begin{prop}\label{prop:hyperbolicextensions}
Let $g:\PP^1(K) \to \RR$ be a continuous function. Given $[\gamma]\in \SL_2(K) / \SU_2(K)$, let
\[
\check{g} ([\gamma]) := \int_{\PP^1(K)} g(\gamma(\alpha)) \omega_{K}\ .
\] Then 
\[
\H\{g\}(\gamma(j)) = \check{g}([\gamma])\ .
\]
\end{prop}

\begin{proof}
By Lemma~\ref{lem:matrixdecomps} each $[\gamma]\in \SL_2(K) / \SU_2(K)$ can be represented by a matrix $$\gamma =\gamma_{z,t} = \begin{pmatrix} \sqrt{t} & z / \sqrt{t} \\ 0 & 1/\sqrt{t} \end{pmatrix}\ ,$$ where $z\in K$ and $t>0$. This, in turn, corresponds to an affine map $\gamma_{z,t}(\alpha) = t\alpha + z$ acting on $\PP^1(K)$. By (\ref{eq:pushforwardomegaK}),
\[
(\gamma_{z,t}^{-1})^* \omega_{K} = P_K(\alpha-z, t) \frac{d\ell_K}{\pi}\ .
\] Therefore,
\begin{align*}
\check{g}([\gamma_{z,t}]) &= \int_{\PP^1(K)} g(\gamma_{z,t}(\alpha)) \omega_{K}\\
& =  \int_{\PP^1(K)} g(\alpha) (\gamma_{z,t}^{-1})^* \omega_{K}\\
& = \int_{K} g(\alpha) P_K(\alpha-z,t) d\ell_K \\
& = (g* P_K)(z,t) = \H\{g\}(\gamma_{z,t}(j))
\end{align*} as asserted. 
\end{proof}

As an application of this -- and because we will need it later -- we have
\begin{prop}\label{prop:logconvergencesphere}
Let $\nu$ be a probability measure on $S^2 = \partial \BB_\CC$ so that the pushforward $\Sigma_* \nu$ by stereographic projection has continuous potentials on $\PP^1(\CC)$, i.e. there is a continuous function $g_\nu:\PP^1(\CC) \to \RR$ so that $dd^c g_\nu = \Sigma_* \nu - \delta_\infty$ as measures on $\PP^1(\CC)$.

For any point $\xi_0 \in S^2$ and any sequence of points $\xi_n \in \BB_{\CC}$ with $\xi_n \to \xi_0$, we have
\[
\int_{S^2} \log |\zeta - \xi_n| d\nu(\zeta) \to \int_{S^2} \log |\zeta - \xi_0| d\nu(\zeta)\ .
\]
\end{prop}

\begin{proof}
By rotating the sphere, we can assume that $\xi_0 = (0, 0, -1)$ is the `south pole'; notice that $\Sigma(\xi_0) = 0\in \CC$. We need the following lemma:

\begin{lemma}
Let $\zeta, \xi\in \overline{\BB_\CC}$ correspond to points $(z,t), (z',t')\in \overline{\hh_{\CC}}$ (resp.), and write $\vn = (0,0,1)$. Then
\[
|\zeta-\xi|^2 = \left\{\begin{matrix}\frac{4\left( |z-z'|^2 +(t'-t)^2\right)}{(|z|^2 + (t+1)^2) \cdot (|z'|^2 + (t'+1)^2)}\ , & \zeta, \xi \neq \vn\\[7pt]
\frac{4}{|z'|^2+(t'+1)^2}\ , & \zeta = (0,0,1), \xi \neq \vn\\[7pt]
\frac{4}{|z|^2 + (t+1)^2}\ , & \zeta \neq (0,0,1), \xi = \vn\\[7pt]
0\ , & \zeta = \xi = \vn\end{matrix}\right.
\] where $|\cdot|$ on the left side is chordal distance, and on the right side is used to denote the usual absolute value on $\CC$.
\end{lemma}
\begin{proof}
Note that the last three cases can be obtained from the first by taking limits as $\zeta\to \vn$ and / or $\xi \to \vn$. So, it suffices to prove the first formula.

The map $\iota_{\CC}$ appearing in Proposition~\ref{prop:usualhypisoms} can be computed explicitly as $\iota_{\CC} = \sigma\circ \eta$, where $\sigma$ is inversion in the sphere centered at $\vn$ of radius $\sqrt{2}$, and $\eta$ is reflection in the hyperplane $\{(x, y, 0)\ : x,y\in \RR\} \subseteq \EE_{\CC}$. Let $p = \eta(z,t) = (z,-t)$ and $q = \eta(z', t') = (z', -t')$. Then by \cite{Rat} Theorem 4.1.3 we have
\[
|\zeta - \xi| = |\sigma(p) - \sigma(q)| = \frac{2 |p-q|}{|p-\vn| \cdot |q-\vn|}\ ,
\] where here all of the absolute values are in terms of the Euclidean metric on $\EE_{\CC}$. Inserting the definition of $p,q$ yields
\[
|\zeta - \xi|^2 = \frac{4(|z-z'|^2 + (t'-t)^2}{(|z|^2 + (t+1)^2)\cdot(|z'|^2 + (t'+1)^2)}\ ,
\] where now on the right side we are using the absolute values to denote the usual distance on $\CC$.
\end{proof}

We now return to the proof of Proposition~\ref{prop:logconvergencesphere}. Note that in our case, for $\zeta, \xi, \xi_0\in \overline{\BB_{\CC}}$ corresponding to $(z,0), (z',t'), (0,0)\in \overline{\hh_{\CC}}$ (resp.), we find
\begin{align*}
|\zeta - \xi|^2 &= \frac{4 (|z-z'|^2 + (t')^2)}{(|z|^2 + 1) \cdot (|z'|^2 + (t'+1)^2)}\\
|\zeta-\xi_0|^2 & = \frac{4|z|^2}{|z|^2+1}\\
|\xi - \xi_0|^2 & = \frac{4 (|z'|^2 + (t')^2)}{|z'|^2 + (t'+1)^2}\ .
\end{align*} In particular,
\begin{equation}
\frac{|\zeta - \xi|^2}{|\zeta-\xi_0|^2} = \frac{|z-z'|^2 + (t')^2}{|z|^2\cdot(|z'|^2 + (t'+1)^2)}\ .
\end{equation}
Taking logarithms, integrating, and doing some simplification yields
\[
\int_{S^2} \log \frac{|\zeta - \xi|}{|\zeta-\xi_0|} d\nu(\zeta)= \frac{1}{2}\int_{\PP^2(\CC)} \log \frac{|z-z'|^2 + (t')^2}{|z|^2} d\Sigma_* \nu(z) - \log(|z'|^2 + (t'+1)^2)\ .
\] Note that $-\log(|z'|^2 + (t'+1)^2) \to 0$ as $\xi \to \xi_0$, since in this case $(z', t') \to (0,0)$. So it suffices to show that the integral in the above expression tends to 0 as $\xi \to \xi_0$. Note that the integrand tends to 0 as $|z|\to \infty$, so that
\[
\frac{1}{2} \int_{\PP^1(\CC)} \log \frac{|z-z'|^2 + (t')^2}{|z|^2} d\Sigma_* \nu(z) = \frac{1}{2} \int_{\PP^1(\CC)} \log\frac{|z-z'|^2 + (t')^2}{|z|^2} dd^c g_\nu(z)\ ,
\] where $g_\nu$ is the potential function referred to in the statement of the proposition. Applying integration by parts to transfer the $dd^c$ to the integrand, we have
\begin{equation}\label{eq:swapddcearlyprop}
\frac{1}{2}\int_{\PP^1(\CC)} \log\frac{|z-z'|^2 + (t')^2}{|z|^2} dd^c g_\nu = \frac{1}{2}\int_{\PP^1(\CC)} g_\nu\  dd^c \left(\log \frac{|z-z'|^2 + (t')^2}{|z|^2}\right)\ .
\end{equation} The $dd^c$ here can be computed explicitly. Letting $\gamma_{z', t'} = \begin{pmatrix} \sqrt{t'} & z'/\sqrt{t'}\\ 0 & 1/\sqrt{t'}\end{pmatrix}$, by (\ref{eq:ddcpsizero}) we find
\[
dd^c \left(\log \frac{|z-z'|^2 + (t')^2}{|z|^2}\right) = 2(\gamma_{z',t'}^{-1})^*\omega_{\CC} -2 \delta_0\ .
\]Inserting this into (\ref{eq:swapddcearlyprop}) and applying Proposition~\ref{prop:hyperbolicextensions} yields
\begin{align*}
\frac{1}{2}\int_{\PP^1(\CC)} g_\nu dd^c\left(\log \frac{|z-z'|^2 + (t')^2}{|z|^2}\right) &= \int_{\PP^1(\CC)} g_\nu\circ \gamma_{z',t'} - g_\nu(0)\\
& = \H\{g_\nu\}(z',t') - g_\nu(0)\ .
\end{align*} Since $\H\{g_\nu\}$ extends continuously to $\overline{\hh_{\CC}}$, giving $g_\nu$ on the boundary $\PP^1(\CC)$, we see that as $\xi\to \xi_0$, $\H\{g_\nu\}(z',t') - g_\nu(0) \to 0$, which finishes the proposition.
\end{proof}

\subsection{Conformal Barycenters}\label{sect:barycentersDE}
In the case that $K=\CC$, we will make use of the conformal barycenter of an admissible measure on $S^2$, an idea originally due to Douady and Earle \cite{DE}. 

A probability measure $\mu$ on $S^2$ will be called admissible if no point $z\in S^2$ has mass $\geq \frac{1}{2}$. In this case, the function $h_\mu(z) = -\frac{1}{2} \int_{S^2} \log \left(\frac{1-|z|^2}{|z-\zeta|^2}\right) d\mu(\zeta)$ on $\BB_\CC$ is convex with respect to the hyperbolic metric and attains a unique minimum at a point $\Bary(\mu)$ of $\BB_\CC$. The barycenter satisfies $\gamma(\Bary(\mu)) = \Bary(\gamma_* \mu)$ for all orientation-preserving automorphisms of $\PP^1(\CC)\simeq S^2$. We need the following Proposition:

\begin{prop}\label{prop:hyplapDE}
Let $\mu$ be a probability measure on $S^2$, and let $h_\mu(z):= -\frac{1}{2}\int_{S^2} \log \left(\frac{1-|z|^2}{|z-\zeta|^2}\right) d\mu(\zeta)$ be the Douady-Earle function on $\BB_\CC$. Then 
\[
\Delta_h^{\BB_\CC}\ h_\mu = 4\ .
\]
\end{prop}
\begin{proof}
Endow $\BB$ with the usual Euclidean coordinates $(x_1, x_2, x_3)$. The hyperbolic Laplacian on $\BB$ can be written as
\[
\Delta_h^{\BB_\CC} = (1-r^2)^2 \Delta_{\textrm{std}} + 2(1-r^2) \sum_{i=1}^3 x_i \frac{\del}{\del x_i}\ ,
\] where $\Delta_{\textrm{std}} = \sum_{i=1}^3 \frac{\del^2}{\del x_i^2}$ is the standard Euclidean Laplacian on $\RR^3$. By direct computation we find that
\[
\Delta_h^{\BB_\CC} (\log (1- |z|^2))= -6-2|z|^2\ ,
\] and for fixed $\zeta\in S^2$ we find
\[
\Delta_h^{\BB_\CC} (\log |z- \zeta|^2) = 2(1-|z|^2)\ .
\] 
The above computations imply that
\[
\Delta_h^{\BB_\CC} (h_\mu) = -\left((-3-r^2) - \int (1-r^2) d\mu(\zeta)\right) = 4\ .
\]
\end{proof}

\section{Basic Properties of $R_F$}\label{sect:bp}

\subsection{$R_F$ is Proper}
Throughout this section, we will work with the quotient space model of hyperbolic space $(\SL_2(K)/\SU_2(K), \dSL)$. 

\begin{thm}\label{thm:growthrates}
Fix a homogeneous lift $F$ of $f$. The map $R_F$ is smooth and proper; in particular, there are constants $C_1(F), C_2(F)$ so that for any $[\gamma]\in \SL_2(K)$ we have
\[
\frac{d-1}{2} \dSL([\gamma], [\textrm{id}]) +\log C_1(F) - d \leq R_F([\gamma]) \leq \frac{d+1}{2} \dSL([\gamma], [\textrm{id}]) + \log C_2(F)\ .
\]
\end{thm}

\begin{lemma}\label{lem:precomp}
Let $\eta_A= \left(\begin{matrix} e^{A/2} & 0 \\ 0 & e^{-A/2}\end{matrix}\right)\in \SL_2(K)$. Then for any $(X,Y)\in K^2$, we find
\[
e^{-|A|/2} ||X,Y|| \leq ||\eta_A (X,Y)|| \leq e^{|A|/2} ||X,Y||\ .
\]
\end{lemma}
\begin{proof}
Note that 
\[
||\eta_A(X,Y)||^2 = e^A |X|^2 + e^{-A}|Y|^2\leq \max(e^A, e^{-A}) (|X|^2+|Y|^2)\ ;
\] taking square roots establishes the upper bound. The lower bound follows by applying the upper bound to $(\tilde{X}, \tilde{Y}) = \eta_A(X,Y)$.
\end{proof}
\begin{lemma}\label{lem:phibounds}
Let $F =[F_0,F_1]$ be a polynomial endomorphism of $K^2$ with $F,G$ homogeneous and coprime of degree $d\geq 1$. There exist constants $C_1(F), C_2(F)$ such that, for any $X,Y\in K^2$,
\begin{equation}\label{eq:basicestimate1}
C_1(F)\cdot ||X,Y||^d \leq ||F(X,Y)|| \leq C_2(F)\cdot ||X,Y||^d\ .
\end{equation} Moreover, the constants $C_1$ and $C_2$ are $\SU_2(K)$-invariant in the sense that $C_1(F^\tau) = C_1(F)$ and $C_2(F^\tau) = C_2(F)$ for any $\tau\in \SU_2(K)$.
\end{lemma}
\begin{proof}
For $X,Y\in K^2$ with $||X,Y|| = 1$, we have
\[
\min_{||\tilde{X},\tilde{Y}|| = 1} ||F(\tilde{X},\tilde{Y})|| \leq ||F(X,Y)|| \leq \max_{||\tilde{X},\tilde{Y}||=1} ||F(\tilde{X}, \tilde{Y})||
\] The general case reduces to this by dividing the expression in (\ref{eq:basicestimate1}) by $||X,Y||^d$ and using the homogeneity of $F$. We take $C_1(F):= \min_{||\tilde{X},\tilde{Y}|| = 1} ||F(\tilde{X},\tilde{Y})||$ and $C_2(F):= \max_{||\tilde{X},\tilde{Y}|| = 1} ||F(\tilde{X},\tilde{Y})||$, which we observe are $\SU_2(K)$ invariant because $\SU_2(K)$ preserves the norm $||\cdot, \cdot||$.
\end{proof}

\begin{proof}[Proof of Theorem~\ref{thm:growthrates}]
The smoothness of $R_F$ can be seen directly from its definition:
\[
R_F([\gamma]) = \int_{\Sphere^3} \log ||F^\gamma(X,Y)|| d\Vol_{\Sphere^3}\ .
\]
For the order estimates, recall that in local coordinates on $\PP^1$ $R_F$ can be expressed as an integral over $\PP^1(K)$ as
\[
R_F(\gamma) = \int_{\PP^1(K)} \log \frac{||F^\gamma(X,Y)||}{||X,Y||^d} \omega_K\ .
\] Write $\gamma = \tau\cdot \eta_A$ with $A>0$, and let $\Psi := F^\tau$ so that
\[
R_F(\gamma) = \int_{\PP^1(K)} \log \frac{||\Psi^{\eta_A}(X,Y)||}{||X,Y||^d} \omega_K\ .
\] Applying Lemma~\ref{lem:precomp} gives
\begin{align}\label{eq:firstRbound}
-\frac{A}{2}+ \int_{\PP^1(K)} \log \frac{||\Psi(\eta_A(X,Y))||}{||X,Y||^d} \omega_K \leq R_F(\gamma) \leq \frac{A}{2} + \int_{\PP^1(K)} \log \frac{||\Psi(\eta_A(X,Y))||}{||X,Y||^d} \omega_K\ .
\end{align} We now apply Lemma~\ref{lem:phibounds} to estimate the integral $\mathcal{I}_1:=\int_{\PP^1(K)} \log \frac{||\Psi(\eta_A(X,Y))||}{||X,Y||^d} \omega_K$ as
\begin{equation}\label{eq:I1bound}
\log C_1(F) + d\int_{\PP^1(K)} \log \frac{||\eta_A(X,Y)||}{||X,Y||^d} \omega_K \leq \mathcal{I}_1 \leq \log C_2(F) +d\int_{\PP^1(K)} \log \frac{||\eta_A(X,Y)||}{||X,Y||^d} \omega_K\\ \ .
\end{equation}
Let
\[
\mathcal{I}_2 := d\int_{\PP^1(K)} \log\frac{||\eta_A(X,Y)||}{||X,Y||}\omega_K\ .
\] In Lemma~\ref{lem:simplestcalculationRF} the integral in $\mathcal{I}_2$ was computed explicitly as
\[
\frac{1}{2}\int_{\PP^1(K)} \log \frac{||\eta_A \cdot (X,Y)||}{||X,Y||} \omega_K = \left\{\begin{matrix} \log(1+e^A) - \frac{A}{2} - \log 2\ , & K = \RR\\ -\frac{1}{2} + \frac{A}{2} \left(\frac{e^{2A} +1}{e^{2A}-1}\right)\ , & K = \CC \end{matrix}\right.\ .
\]
When $K=\RR$, note that
\[
A \leq \log (1+e^A) \leq A + \log 2\ , 
\] whereby we obtain
\begin{equation}\label{eq:I2boundreal}
(A -2\log 2)\frac{d}{2}\leq \mathcal{I}_2 \leq A\frac{d}{2}\ .
\end{equation}
When $K=\CC$, note that for $A>0$ we have 
\[
A \leq \frac{A(e^{2A}+1)}{e^{2A}-1}\leq A+1
\] From this, we find that
\begin{equation}\label{eq:I2boundcomplex}
(A-1)\frac{d}{2} \leq \mathcal{I}_2 \leq A\frac{d}{2}\ .
\end{equation}

Combining the estimates in (\ref{eq:firstRbound}), (\ref{eq:I1bound}), (\ref{eq:I2boundcomplex}), and (\ref{eq:I2boundreal}), we obtain the asserted bounds for $A>0$:
\[
\frac{d-1}{2}A + \log C_1(F) -d \leq R_F(\gamma) \leq \frac{d+1}{2}A+ \log C_2(F)\ .
\] Recalling that $A = \dSL([\gamma], [\textrm{id}])$, we see that we are done.
\end{proof}

\begin{cor}
$R_F$ attains a minimum.
\end{cor}
\begin{proof}
Let $M=R_F([\textrm{id}])$, and set $e=\frac{2}{d-1}(M-\log C_1(F))$. Theorem~\ref{thm:growthrates} implies that for $\dSL([\gamma], [\textrm{id}]) > e$ we have $R_F([\gamma]) > M$ (note that the constant $C_1(F)$ is $\SU_2(K)$ invariant). In particular, $R_F$ must attain a minimum on the compact set $\overline{B_{\dSL}([\textrm{id}], e)}$, which is necessarily less than or equal to $M$; off of this compact set, we know that $R_F$ is strictly larger than $M$. Thus, the minimum attained on $\overline{B_{\dSL}([\textrm{id}], e)}$ is a global minimum.
\end{proof}

\begin{defn}
Let $f\in K(z)$, and let $F$ be a homogeneous lift of $f$. Then the min-invariant of $f$ is the quantity
\[
m_K(f) = \min_{[\gamma]\in \SL_2(K) / \SU_2(K)} R_F([\gamma]) - \frac{1}{2d} \log |\Res(F_0, F_1)|\ .
\]
\end{defn}

It is expected that the min-invariant is related to the multipliers of the fixed points of $f$. This can be shown explicitly in the case that $d=1$; see Section~\ref{sect:done} below. In the non-Archimedean setting, the connection between the minimal value of $\ord\Res_f$ and the multipliers of $f$ is also known to hold for quadratic rational maps \cite{DJR} and cubic polynomials \cite{HMRVW}.

Also in the non-Archimedean setting, the function $\ord\Res_f$ -- which we recall is the analogue of $R_F$ in that context -- is minimized on what Rumely calls the `minimal resultant locus', denoted $\MinResLoc(f)$. Rumely shows that this set is either a single point or a segment (\cite{Ru1} Theorem 1.1), and that conjugates attaining this minimum correspond to maps with semistable reduction (in the sense of GIT).

The following Theorem shows that the sublevel sets of $R_F$ are bounded, and will be used in Section~\ref{sect:convergenceofmins} below (compare with \cite{Ru1} Theorem 1.1):
\begin{cor}\label{cor:levelbounded}
Suppose that the degree of $f$ is $d\geq 2$. The functions $R_{F^n}$ are level bounded, i.e. for any $\alpha\in \RR$, there is an $R>0$ depending only on $\alpha$ and $F$ such that 
\[
\{[\gamma]\in \SL_2(K) / \SU_2(K)\ : \ R_{F^n}([\gamma]) \leq \alpha\}\subseteq B_R([\textrm{id}])
\] for all $n$. Here, $B_\epsilon([\gamma])$ is the $\epsilon$-ball around $[\gamma]$ in $\SL_2(K)/\SU_2(K)$ with respect to the metric $\dSL$. 
\end{cor}
\begin{proof}
The lower bound in Theorem~\ref{thm:growthrates} for the family is
\begin{equation}\label{eq:repeatlowerbound}
\frac{d^n-1}{2} \dSL([\gamma], [\textrm{id}]) +\log C_1(F^n) \leq R_{F^n}([\gamma])\ ,
\end{equation} where $C_1(F^n) = \min_{||\tilde{X},\tilde{Y}|| = 1} ||F^n(\tilde{X},\tilde{Y})||$. We first observe that that 
$$
C_1(F^n) \geq C_1(F)^{\frac{d^n-1}{d-1}}\ ,
$$ which follows inductively from the fact that
\[
||F^n(X,Y)|| = ||F(F^{n-1}(X,Y))|| \geq C_1(F)\cdot ||F^{n-1}(X,Y)||^d\ .
\] Thus (\ref{eq:repeatlowerbound}) becomes
\[
\frac{d^n-1}{2} \dSL([\gamma], [\textrm{id}]) + \frac{d^n-1}{d-1} \log C_1(F) \leq R_{F^n}([\gamma])\ .
\]
Now let $\alpha\in \RR$. If $[\gamma]\in \SL_2(K)/\SU_2(K)$ is chosen so that $R_{F^n}([\gamma]) \leq \alpha$, then
\begin{align*}
\dSL([\gamma], [\textrm{id}]) & \leq \frac{2}{d^n-1} \alpha - \frac{1}{2(d-1)} \log C_1(F)\\
& \leq 2\max(0, \alpha) - \frac{1}{2(d-1)} \log C_1(F)\ .
\end{align*} Setting $R=2\max(0, \alpha) - \frac{1}{2(d-1)} \log C_1(F)$, the lemma is proved.

\end{proof}

\subsection{$R_F$ is subharmonic}

In this section, it will be most convenient to work with the upper half-space model $(\hh_K, \dH)$ of hyperbolic space. We will show
\begin{thm}\label{thm:hypLap}
For any $[\gamma]\in \SL_2(K)/\SU_2(K)$, we have 
\[
\Delta_h^{\hh_K} R_F([\gamma]) = [K:\RR]\left(d-1+ 4\int_{\PP^1(K)}  || f^\gamma(w), w||^2 \omega_K\right)\ ,
\] where $\Delta_h$ is the Laplace-Beltrami operator for $(\hh_K, \dH)$. In particular, $R_F$ is strictly subharmonic for $F$ not equal to the identity map. 
\end{thm}
The proof rests on several explicit calculations. Recall first that we showed in Proposition~\ref{prop:RFLandP} that
\begin{equation}\label{eq:RFaffineforlap}
R_F([\gamma]) = \frac{1}{2}\int_{\PP^1(K)} \log \frac{|f_1(w)|^2 \cdot \left( |f(w) - z|^2 + t^2\right)}{t^{d+1}} P_K(\alpha-z, t) \frac{d\ell_K}{\pi} +c_K
\end{equation} for an explicit constant $c_K$. Let 
$$
L(z,t;w) = \log(|f(w)-z|^2 + t^2) + 2\log |f_1(w)| - (d+1)\log t
$$ be the logarithmic component of the integrand above. We begin by computing its hyperbolic Laplacian:

\begin{lemma} \label{lemma:derivativesofL}Fix $w\in K$. Let $z=z_0+iz_1$ be a complex number, and write $f(w) = f_{\textrm{Re}} + i \fim$. If $K=\RR$, we assume $z=z_0, f(w) = \fre$. 
\begin{align*}
\frac{\partial}{\partial t} L(z,t;w) &= \frac{2t}{|f(w) - z|^2 + t^2} - \frac{d+1}{t}\\
\frac{\partial}{\partial z_0} L(z,t;w) & = \frac{-2( \fre - z_0)}{|f(w) - z|^2 + t^2}\\
\frac{\partial}{\partial z_1} L(z,t;w) & = \frac{-2( \fim - z_1)}{| f(w) - z|^2 + t^2}\\
\frac{\partial^2}{\partial t^2} L(z,t;w) &= \frac{2 |f(w) - z|^2 -2t^2}{(|f(w) - z|^2 + t^2)^2} + \frac{d+1}{t^2}\\
\frac{\partial^2}{\partial z_0^2} L(z,t;w) & = \frac{2( \fim - z_1)^2 - 2( \fre - z_0)^2+2t^2}{(| f(w) - z|^2 + t^2)^2}\\
\frac{\partial^2}{\partial z_1^2} L(z,t;w) & = \frac{ 2( \fre - z_0)^2 -2 ( \fim - z_1)^2 + 2t^2}{(| f(w) - z|^2 + t^2)^2}\ .
\end{align*} In particular, 
\[
\Delta_h^{\hh_K} L(z,t;w) = [K:\RR](d+1)\ .
\]
\end{lemma}
\begin{proof}
Write $| f(w) - z|^2 = ( \fre-z_0)^2 + ( \fim - z_1)^2$, where we note that there is no $\fim-z_1$ term in the case $K=\RR$. The formulas for the derivatives follow directly from basic calculus differentiation rules. We omit the first derivative calculations, and compute $\deldeltwo{t} L(z,t;w)$ and $\deldeltwo{z_0} L(z,t;w)$:
\begin{align*}
\deldeltwo{t} L(z,t;w) & = \deldel{t} \deldel{t} L(z,t) = \deldel{t}\left( \frac{2t}{|f(w) - z|^2 + t^2} - \frac{d+1}{t}\right)\\
& = \frac{2(| f(w)-z|^2 + t^2) - (2t)^2}{(| f(w)-z|^2 + t^2)^2} + \frac{d+1}{t^2}\\
& = \frac{ 2| f(w)-z|^2 - 2t^2}{(| f(w)-z|^2 + t^2)^2} + \frac{d+1}{t^2}\ .
\end{align*}

For $\deldeltwo{z_0}L(z,t;w)$:
\begin{align*}
\deldeltwo{z_0}L(z,t;w) & = \deldel{z_0} \left(\frac{-2( \fre - z_0)}{| f(w)-z|^2 + t^2}\right)\\
& = \frac{ 2(| f(w) - z|^2 + t^2) +2( \fre-z_0) (-2( \fre - z_0))}{(| f(w) - z|^2 + t^2)^2}\\
& = \frac{2t^2 + 2( \fim - z_1)^2 -2 ( \fre - z_0)^2}{(| f(w) - z|^2 + t^2)^2}\ .
\end{align*} The case of $\deldeltwo{z_1}L(z,t;w)$ is similar.

We now apply the definition of the hyperbolic Laplacian. If $K=\RR$:
\begin{align*}
\Delta_h^{\hh_K} L(z,t;w) & = t^2 \cdot \left(\deldeltwo{z_0} + \deldeltwo{t}\right) L(z,t)\\
& = t^2 \left( \frac{-2( \fre - z_0)^2 + 2t^2}{(| f(w) - z|^2 + t^2)^2} + \frac{2| f(w) - z|^2 - 2t^2}{(| f(w) - z|^2 + t^2)^2} + \frac{d+1}{t^2}\right)\ .
\end{align*}Note that in this case, $( \fre - z_0)^2 = | f(w) - z|^2$, so that the above formula reduces to $\Delta_h^{\hh_K} L(z,t;w) = d+1=[K:\RR](d+1)$ as claimed.
In the case $K=\CC$, the expression for the Laplacian is more involved:
\begin{align*}
\Delta_h^{\hh_K} L(z,t;w) & = t^2 \left(\deldeltwo{z_0} + \deldeltwo{z_1} + \deldeltwo{t}\right) L(z,t;w) - t \deldel{t} L(z,t;w)\\
& = t^2 \left(\frac{2( \fim - z_1)^2 - 2( \fre - z_0)^2+2t^2}{(| f(w) - z|^2 + t^2)^2}+ \frac{ 2( \fre - z_0)^2 -2 ( \fim - z_1)^2 + 2t^2}{(| f(w) - z|^2 + t^2)^2}\right.\\
& \left. +\frac{2 | f(w) - z|^2 -2t^2}{(| f(w) - z|^2 + t^2)^2} + \frac{d+1}{t^2}\right)- t \cdot \left(\frac{2t}{| f(w) - z|^2 + t^2} - \frac{d+1}{t}\right)\\
& = t^2 \cdot \frac{2| f(w) - z|^2 + 2t^2}{(| f(w) - z|^2 + t^2)^2} + (d+1) - \frac{2t^2}{(| f(w)-z|^2 + t^2)}+(d+1)\\
& = 2(d+1) \\
&=[K:\RR](d+1)\ .
\end{align*}
\end{proof}

We next compute the derivatives and hyperbolic Laplacian of $P_K(w -z,t)$:
\begin{lemma}\label{lem:derivativesofP}Write $w= w_0 + i w_1$, $z= z_0 + i z_1$, where we understand $w_1 = z_1 = 0$ when $K=\RR$. Then 
\begin{align*}
\deldel{t} P_K(w-z,t) & = [K:\RR] \cdot \frac{t^{[K:\RR]-1} (|w-z|^2 - t^2)}{(|w-z|^2 + t^2)^{1+[K:\RR] }}\\
\deldel{z_0} P_K (w-z,t) & = [K:\RR] \cdot \frac{2t^{[K:\RR]}(w_0 - z_0)}{(t^2 + |w-z|^2)^{1+[K:\RR] }}\\
\deldel{z_1} P_K(w-z,t) & =[K:\RR] \cdot \frac{2t^{[K:\RR]}(w_1 - z_1)}{(t^2 + |w-z|^2)^{1+[K:\RR]}}\\
\frac{\partial^2}{\partial t^2} P_K(w - z,t)  & = [K:\RR]([K:\RR]+1) \cdot \frac{t^{[K:\RR]+2} - 3t^{[K:\RR]}|w-z|^2}{(t^2+|w-z|^2)^{[K:\RR]+2}}\\
\frac{\partial^2}{\partial z_0^2} P_K(w - z,t) & = [K:\RR] \frac{(4[K:\RR]+2)t^{[K:\RR]}(w_0-z_0)^2 - 2t^{[K:\RR]}(w_1-z_1)^2 - 2t^{[K:\RR]+2}}{(t^2+|w-z|^2)^{2+[K:\RR]}}\\
\frac{\partial^2}{\partial z_1^2} P_K(w - z,t) & =[K:\RR]\frac{(4[K:\RR]+2)t^{[K:\RR]}(w_1-z_1)^2 - 2t^{[K:\RR]}(w_0-z_0)^2 - 2t^{[K:\RR]+2}}{(t^2+|w-z|^2)^{2+[K:\RR]}} \\
\end{align*}

In particular,
\[
\Delta_h^{\hh_K} P_K(w - z,t) = 0\ .
\]
\end{lemma}
\begin{proof}
Here again, the result is straightforward calculus. For $t$, the first derivative is
\begin{align*}
\deldel{t} P_K(w-z,t) & =  \deldel{t} \left(\frac{t}{t^2 + |w-z|^2}\right)^{[K:\RR]}\\
& = [K:\RR] \cdot \left( \frac{t}{t^2 + |w-z|^2}\right)^{[K:\RR]-1} \cdot \frac{t^2 + |w-z|^2 - t(2t)}{(t^2 + |w-z|^2)^2}\\
& = [K:\RR] \cdot \frac{t^{[K:\RR]-1} (|w-z|^2 - t^2)}{(t^2 + |w-z|^2)^{1+[K:\RR]}}\ .
\end{align*}

The first derivative for $z_0$ is
\begin{align*}
\deldel{z_0} P_K(w-z,t) & = \deldel{z_0}\left(\frac{t}{t^2 + |w-z|^2}\right)^{[K:\RR]}\\
& =[K:\RR] \cdot \left( \frac{t}{t^2 + |w-z|^2}\right)^{[K:\RR]-1} \cdot \frac{2t(w_0 - z_0)}{(t^2 + |w-z|^2)^2}\\
& =[K:\RR] \cdot \frac{2t^{[K:\RR]} (w_0-z_0)}{(t^2 + |w-z|^2)^{1+[K:\RR]}}\ ;
\end{align*} the calculation of the first derivative for $z_1$ is symmetric. The calculations for the second derivatives are tedious but straightforward. The fact that $\Delta_h^{\hh_K} P_K(w - z,t)=0$ now follows from the calculations of the derivatives.
\end{proof}

In order to compute the hyperbolic Laplacian of $R_F$, we pass the operator $\Delta_h^{\hh_K}$ into the integral defining $R_F$ and compute the Laplacian of the integrand, which is a product $L(z,t;w)\cdot P_K(w-z,t)$. Thus we will require a product formula for the Laplacian (see, e.g. **):
\begin{equation}\label{eq:hyplapprodform}
\Delta_h^{\hh_K} (fg) = g\ \Delta_h^{\hh_K}f + 2t^2 \nabla_\textrm{std} f \cdot \nabla_\textrm{std} g + f\ \Delta_h^{\hh_K}g \ .
\end{equation} We apply this to $L(z,t;w)P_K(w-z,t)$:
\begin{prop}\label{prop:hypLapLP}
Let $z\in K$ and $t>0$, and let $\gamma(w) = \gamma_{z,t}(w) = tw+z$. The hyperbolic Laplacian of $L(z,t;w)P_K(w-z,t)$ is given
\begin{equation}\label{eq:hypLapLP}
\Delta_h^{\hh_K} \left(L(z,t;w)P_K(w-z,t)\right) =  [K:\RR] P_K(w-z,t)\left( (d-1) + \frac{4dt^2}{|w-z|^2+t^2} + 4||
\gamma^{-1}( f(w)), \gamma^{-1}(w)||^2\right)\ ,
\end{equation} where we recall that $||\cdot, \cdot||$ is the chordal distance on $\PP^1(K)$. 
\end{prop}
\begin{proof}
First, recall that $\Delta_h^{\hh_K} L(z,t;w) = [K:\RR](d+1)$ (Lemma~\ref{lemma:derivativesofL}) and $\Delta_h^{\hh_K} P_K(w-z,t) = 0$ (Lemma~\ref{lem:derivativesofP}). Inserting this into (\ref{eq:hyplapprodform}) yields
\begin{equation}\label{eq:hyplapLPfirst}
\Delta_h^{\hh_K}\left(L(z,t;w) P_K(w-z,t)\right) = [K:\RR](d+1) P_K(w-z,t) + 2t^2 \nabla_\textrm{std} L(z,t;w) \cdot \nabla_\textrm{std} P_K(w-z,t)\ .
\end{equation}

We next compute the product of the gradients in the above expression. From Lemma~\ref{lemma:derivativesofL} we find
\[
\nabla_\textrm{std} L(z,t;w) = \frac{1}{| f(w)-z|^2 + t^2} \cdot \langle -2( \fre-z_0), -2( \fim - z_1), -(d-1)t - t^{-1}(d+1)| f(w) - z|^2\rangle\ ,
\]where there is no $z_1$ component when $K=\RR$. Similarly, from Lemma~\ref{lem:derivativesofP} we have
\[
\nabla_\textrm{std} P_K(w-z,t) = \frac{[K:\RR]\cdot t^{[K:\RR]}}{(|w-z|^2 + t^2)^{1+[K:\RR]}} \langle 2(w_0-z_0), 2(w_1-z_1), t^{-1} \cdot(|w-z|^2 - t^2)\rangle\ .
\] The dot product we are interested in can then be rewritten as
\begin{equation}\label{eq:prodofgradients}
\nabla_\textrm{std} L(z,t;w) \cdot \nabla_\textrm{std}P_K(w-z,t)= [K:\RR]\cdot P_K(w-z,t) \cdot \frac{1}{(| f(w)-z|^2 +t^2)\cdot (|w-z|^2 + t^2)} \vv_L\cdot \vv_P\ ,
\end{equation} where $\vv_L, \vv_P$ are the vector parts of the respective gradients. Explicitly
\begin{align*}
\vv_L \cdot \vv_P &= -4( \fre - z_0)(w_0 - z_0) - 4( \fim - z_1) (w_1 - z_1) + t^{-1}(-(d-1)t- t^{-1}(d+1) | f(w)-z|^2)(|w-z|^2 - t^2)\ .
\end{align*} We note that 
\begin{align*}
2( \fre-z_0) (w_0 - z_0) + 2( \fim - z_1) (w_1 - z_1) &= 2\textrm{Re}\left( ( f(w) - z)\cdot (\overline{w-z})\right)\\
& = | f(w) - z|^2 + |w-z|^2 - |( f(w) -z) - (w-z)|^2\\
& = | f(w) - z|^2 + |w-z|^2 - | f(w) - w|^2\ .
\end{align*} Inserting this into the expression for $\vv_L \cdot \vv_P$ and simplifying gives
\begin{align*}
\vv_L \cdot\vv_P  =& -2| f(w) - z|^2 -2 |w-z|^2 +2 | f(w) - w|^2 + t^{-1}\left(-(d-1)t - t^{-1}(d+1) | f(w)-z|^2\right)\left(|w-z|^2 - t^2\right)\\
=& -(d+1)t^{-2}\left(| f(w)-z|^2 + t^2\right)\left(|w-z|^2 + t^2\right) + 2| f(w)-w|^2 + 2d(t^2+| f(w)-z|^2)\ .
\end{align*} Inserting this into the expression for the product of the gradients given in (\ref{eq:prodofgradients}) and simplifying, we find
\begin{align*}
\nabla_\textrm{std} L(z,t;w)& \cdot \nabla_\textrm{std}P_K(w-z,t)\\
& =[K:\RR]\cdot P_K(w-z,t)\left(-(d+1)t^{-2} + \frac{2d}{|w-z|^2 + t^2} + \frac{2| f(w)-w|^2}{(| f(w)-z|^2+t^2)(|w-z|^2+t^2)} \right)\ .
\end{align*} Now plugging this into (\ref{eq:hyplapLPfirst}) gives
\begin{align}
\Delta_h^{\hh_K} \left(L(z,t;w) P_K(w-z,t)\right) = & [K:\RR](d+1)P_K(w-z,t) +2t^2 \nabla_\textrm{std} L(z,t;w) \cdot\nabla_\textrm{std} P_K(w-z,t)\nonumber \\
=& [K:\RR] P_K(w-z,t)\left( -(d+1) + \frac{4dt^2}{|w-z|^2+t^2} + \frac{4t^2| f(w)-w|^2}{(| f(w)-z|^2 + t^2)(|w-z|^2+t^2)}\right)\ .\nonumber
\end{align} We are done, noticing that the last term can be manipulated as follows
\[
\frac{t^2| f(w)-w|^2}{(| f(w) -z |^2+t^2)(|w-z|^2+t^2)} = \frac{\left|\frac{ f(w) - z}{t} - \frac{w-z}{t}\right|^2}{\left(\left|\frac{ f(w) -z}{t}\right|^2 + 1\right)\left(\left|\frac{w-z}{t}\right|^2 + 1\right)} = ||\gamma^{-1}( f(w)), \gamma^{-1}(w)||^2\ .
\]
\end{proof}

We are now ready to prove Theorem~\ref{thm:hypLap}:
\begin{proof}[Proof of Theorem~\ref{thm:hypLap}]
We begin by noting that, in computing $\Delta_h^{\hh_K} R_F$, we can pass the Laplacian under the integral sign appearing in (\ref{eq:RFaffineforlap}) since all of the derivatives of $L,P$ are smooth functions of $(z,t)\in \hh_K$. Thus, we will compute
\[
\Delta_h^{\hh_K} R_F([\gamma]) = \int_{\PP^1(K)} \Delta_h^{\hh_K} L(z,t;w) P_K(w-z, t) \frac{d\ell_K}{\pi}\ .
\]
Here we will use the formula derived in Proposition~\ref{prop:hypLapLP}. Note that $\int_{\PP^1(K)} P_K(w-z,t) \frac{d\ell_K}{\pi} = \int_{\PP^1(K)} (\gamma_{z,t}^{-1})^* \omega_K = 1$, and
\begin{align*}
\int_{\PP^1(K)} \frac{t^2}{|w-z|^2 + t^2} P_K(w-z,t) \frac{d\ell_K}{\pi} & = \int_{\PP^1(K)} \frac{1}{|\gamma_{z,t}^{-1}(w)|^2 + 1} (\gamma_{z,t}^{-1})^* \omega_K\\
& = \int_{\PP^1(K)} \frac{1}{|w|^2 + 1} \omega_K = 1\ .
\end{align*} Thus, integrating the expression in (\ref{eq:hypLapLP}) against $ \frac{d\ell_K}{\pi}$ and simplifying gives
\[
\int_{\PP^1(K)} \Delta_h^{\hh_K} L(z,t;w)P_K(w-z,t)\ \omega_K(w) = [K:\RR]\left((d-1) + 4 \int_{\PP^1(K)} ||f^\gamma(w), w||^2 \omega_K(w)\right)\ ,
\] which finishes the proof.
\end{proof}

The calculations of the preceeding sections are enough to prove Theorem~\ref{thm:basicprops}:
\begin{proof}[Proof of Theorem~\ref{thm:basicprops}:]
The fact that $R_F$ is smooth on $\hh_K$ follows from the expression for $R_F$ given in Proposition~\ref{prop:RFLandP}. Properness follows from Theorem~\ref{thm:growthrates} since $d\geq 2$; likewise, subharmonicity follows from Theorem~\ref{thm:hypLap}. 
\end{proof}

Note that $R_F$ is subharmonic even when $d=1$, provided $f$ is not the identity map. However, it need not be proper when $d=1$; see Section~\ref{sect:done}.

\subsection{Geometry and Minimizers}
In this section we give a geometric interpretation to the conjugates $f^\gamma$ for which $R_F([\gamma])$ is minimized, in the case that $K=\CC$. Note that, viewing $\omega_\CC$ as a probability measure on $\PP^1(\CC)$, we can define both $f^* \omega_\CC$ and $f_* \omega_\CC$. We let
\[
\omega_f:= f^* \omega_\CC + f_* \omega_\CC\ .
\] It is a positive measure of total mass $d+1$ on $\PP^1(\CC)$. Our main result in this section is the following:
\begin{thm}\label{thm:minimizers}
Let $f\in \CC(z)$, and let $\gamma\in \SL_2(\CC) / \SU_2(\CC)$ be a conjugate for which $R_F(\cdot)$ is minimized. Then the barycenter of the measure $\widehat{\omega_{f^\gamma}}$ on $S^2$ is $0\in \BB_\CC$, where $\widehat{\omega_{f^\gamma}}$ is the pullback of $\omega_{f^\gamma}$ from $\PP^1(\CC)$ to $S^2$ via stereographic projection. 
\end{thm}

This theorem will be an immediate consequence of the following theorem:
\begin{thm}\label{thm:gradientRF}
Let $f\in \CC(z)$, and view $R_F$ as a function on $\BB_\CC$. Then
\[
\nabla_h R_F(\xi) = -\int_{S^2} \zeta\ \widehat{\omega_{f^\gamma}}(\zeta)\ ,
\] where $\xi\in \BB_\CC$ corresponds to the class $[\gamma]\in \SL_2(\CC) / \SU_2(\CC)$.
\end{thm}

\noindent{\em Remark:} While the results above are stated for $R_F$ as a function on $\BB_\CC$, we will heavily rely on expressions for $R_F$ on $\hh_{\CC}$ and $\SL_2(\CC)/\SU_2(\CC)$ in proving this theorem.\\

To begin, we recall the expression for $R_F$ derived in Section~\ref{sect:expressions} (see (\ref{eq:simplifiedRF}); the precise expressions for $\psi, \omega_{z,t}$ and $c(z,t;f)$ appear in Section~\ref{sect:expressions}):
\[
R_F([\gamma]) = \frac{d-1}{2} \log t + \int_{\PP^1(\CC)} \psi(\alpha; z,t) f^*  \omega_{z,t} - \frac{d}{2} + \frac{1}{2} \log c(z,t;f)\ .
\]

With this expression, we can explicitly determine a directional derivative of $R_F$:
\begin{prop}\label{prop:dirderivRF}
Let $\vv_\infty\in T_{0}\BB_{\CC}$ be the direction towards $N=(0,0,1)\in S^2$. Then 
\[
\partial_{\vv_\infty} R_F(0) = \frac{1}{4}\int_{S^2} |\zeta - N|^2 - |\zeta + N|^2\ \widehat{\omega_f}(\zeta)\ .
\]
\end{prop}
\begin{proof}
The direction $\vv_\infty\in T_0$ can be represented by the path $\{(0,0, t)\ : \ t>0\}$, where $t$ is viewed as a parameter in the hyperbolic metric on $\BB_{\CC}$. This path, in turn, corresponds to the path of matrices $\{[\eta_A] \ : \ A >0\} \in \SL_2(\CC) / \SU_2(\CC)$, where 
\[
\eta_A  = \left(\begin{matrix} e^{A/2} & 0 \\ 0 & e^{-A/2} \end{matrix}\right)\ .
\] In particular,
\[
\lim_{t\to 0} \frac{R_F( (0,0,t)) - R_F((0,0,0))}{t} = \lim_{A\to 0} \frac{R_F([\eta_A]) - R_F([\textrm{id}])}{A} = \partial_A \left(R_F([\eta_A])\right)|_{A=0}\ .
\] We can compute the latter derivative with the help of (\ref{eq:simplifiedRF}). Note that
\begin{equation}\label{eq:RFA}
R_F([\eta_A]) = \frac{d-1}{2} A + \int_{\PP^1(\CC)} \psi(\alpha; 0, e^A) f^* \omega_{0, e^A} - d + \frac{1}{2} \log c(0, e^A; f)\ .
\end{equation} We easily see that $\partial_A \left( \frac{d-1}{2} A\right)= \frac{d-1}{2}$, and a straightforward calculation shows that
\[
\psi_A(\alpha; 0, e^A) := \partial_A \psi(\alpha; 0, e^A) =\frac{-|\alpha|^2}{e^{2A} + |\alpha|^2}\ .
\] and that the Laplacian of $\psi_A$ in the $\alpha$ coordinate is
\begin{align*}
\omega_A := \frac{i}{2\pi} \partial_{\overline{\alpha}}\partial_\alpha \left(\frac{-|\alpha|^2}{e^{2A} + |\alpha|^2}\right)d\alpha d\overline{\alpha} &= \frac{i}{2\pi} \frac{1}{(1+|\eta_A^{-1}(\alpha)|^2)^2}\cdot \frac{|\eta_A^{-1}(\alpha)|^2 -1}{|\eta_A^{-1}(\alpha)|^2 + 1} d\alpha d\overline{\alpha}\\
& = \frac{|\eta_A^{-1}(\alpha)|^2 -1}{|\eta_A^{-1}(\alpha)|^2 + 1} \omega_{0,e^A}\ .
\end{align*} Integrating $\omega_A$ over $\PP^1(\CC)$ gives 0, and we conclude $dd^c \psi_A(\cdot; 0,e^A) = \omega_A$. Lastly, note that
\[
\partial_A \frac{1}{2}\log c(0, e^A; f) = \frac{e^{2A}|b_d|^2}{ |a_d|^2 + e^{2A}|b_d|^2} = \frac{e^{2A}}{|f(\infty)|^2 + e^{2A}}\ ,
\]which is $0$ if $f(\infty) = \infty$.

We are now ready to compute $\partial_A R_F$. Differentiating (\ref{eq:RFA}) in $A$ gives
\begin{align}
\partial_A R_F& = \frac{d-1}{2} + \int_{\PP^1(\CC)} \psi_A(\alpha; 0, e^A) f^* \omega_{0,e^A} + \psi(\alpha; 0, e^A) f^* \omega_A + \frac{e^{2A}}{|f(\infty) |^2 + e^{2A}}\nonumber \\
& = \frac{d-1}{2} + \int_{\PP^1(\CC)} \psi(\alpha; 0, e^A) f^*  \omega_{0,e^A} + \psi(\alpha; 0, e^A) f^* dd^c (\psi_A) + \frac{e^{2A}}{|f(\infty)|^2 + e^{2A}}\nonumber \\
& = \frac{d-1}{2} + \int_{\PP^1(\CC)} \psi_A(\alpha; 0, e^A) f^*  \omega_{0,e^A} + \int_{\PP^1(\CC)} f^* \psi_A(\alpha; 0, e^A) dd^c (\psi) + \frac{e^{2A}}{|f(\infty)|^2+e^{2A}}\label{eq:firststepderivativeRF}\ ,
\end{align} where in the last step we have used the pullback formula for $dd^c$ and then applied integration by parts to move the $dd^c$ onto the integrand. Recall that $dd^c \psi =  \omega_{0,e^A} - \delta_\infty$, so that 
\[
\int_{\PP^1(\CC)} f^*\psi_A(\alpha; 0, e^A) dd^c \psi = \int_{\PP^1(\CC)} f^* \psi_A(\alpha; 0, e^A)  \omega_{0,e^A} - \psi_A(f(\infty); 0, e^A)\ .
\] Inserting this into (\ref{eq:firststepderivativeRF}) and simplifying yields
\begin{align}
\partial_A R_F &= \frac{d+1}{2} +\int_{\PP^1(\CC)} \psi_A\ f^*  \omega_{0,e^A} + f^* \psi_A\  \omega_{0,e^A}\nonumber\\
& = \frac{d+1}{2} +\int_{\PP^1(\CC)} \psi_A (f^*  \omega_{0,e^A} + f_*  \omega_{0,e^A})\label{eq:secondstepderivativeRF}\ ,
\end{align} where we now interpret $f^*  \omega_{0,e^A}, f_*  \omega_{0,e^A}$ as the pullback / pushforward {\em in the sense of measures on $\PP^1(\CC)$}. We define $\omega_{0,e^A;f} = f^* \omega_{0,e^A} + f_*  \omega_{z,e^A}$; it is a measure of total mass $d+1$ on $\PP^1(\CC)$. 

Now, recall that $\psi_A(\alpha; 0, e^A) = \frac{-|\alpha|}{e^{2A} + |\alpha |^2} = \frac{-|\eta_A^{-1}(\alpha)|^2}{1+|\eta_A^{-1}(\alpha)|^2}$. Since the measure $\omega_{0,e^A;f}$ has total mass $d+1$, the expression in (\ref{eq:secondstepderivativeRF}) can be written as
\begin{align}
\partial_A R_F &= \int_{\PP^1(\CC)} \frac{1}{2} - \frac{|\eta_A^{-1}(\alpha)|^2}{1+|\eta_A^{-1}(\alpha)|^2}\ \omega_{0,e^A;f}\nonumber\\
& = \int_{\PP^1(\CC)} \frac{1 - |\eta_A^{-1}(\alpha)|^2}{1+|\eta_A^{-1}(\alpha)|^2}\ \omega_{0,e^A; f}\nonumber\\
& = \int_{\PP^1(\CC)} ||\eta_A^{-1}(\alpha), \infty||^2 - ||\eta_A^{-1}(\alpha), 0||^2\ \omega_{0,e^A; f}\nonumber\\
& = \int_{\PP^1(\CC)} ||\alpha, \infty||^2 - ||\alpha, 0||^2\ (\eta_A)_* \omega_{0, e^A; f}\label{eq:laststepderivativeRF}\ .
\end{align} We are now essentially done: recall that $ \omega_{0,e^A} = (\eta_A^{-1})^* \omega_\CC = (\eta_A)_* \omega_\CC$. It is straightforward to check then that $(\eta_A)_* \omega_{z,e^A;f} = (f^{\eta_A})^* \omega_\CC + (f^{\eta_A})_* \omega_\CC = \omega_{f^{\eta_A}}$; setting $A=0$ and inserting this into (\ref{eq:laststepderivativeRF}) gives
\[
\partial_A R_F|_{A=0} = \int_{\PP^1(\CC)} ||\alpha, \infty||^2 - ||\alpha, 0||^2 \omega_f(\alpha)\ .
\] Pulling this back to $S^2$ under stereographic projection, and noticing that $||z,w|| =\frac{1}{2} |\hat{z} - \hat{w}|$ yields the statement in the proposition (here, $\hat{z}, \hat{w}$ are the pullback to $S^2$ under stereographic projection).
\end{proof}

\noindent{\em Remark:} Using the previous Proposition, one can give a general formula for the directional derivative in any direction at any point $\xi\in \BB_\CC$; the formula is not particularly enlightening, and is not needed in what follows, so we have omitted it.

\begin{lemma}\label{lem:barycentercriteria}
Let $\mu$ be a positive measure of finite volume on $S^2$. Then for any $a\in \RR^3$, we have
\[
\int_{S^2} |\zeta - a|^2 - |\zeta + a|^2 d\mu(\zeta) = -4\left(\int_{S^2} \zeta d\mu(\zeta)\right) \cdot a\ .
\] In particular, the following are equivalent:
\begin{enumerate}
\item The conformal barycenter of $\mu$ is $0\in \RR^3$.
\item For every $a\in \RR^3$, $\int_{S^2} |\zeta - a|^2 - |\zeta + a|^2 d\mu(\zeta) = 0$
\item There exist three linearly independent vectors $a_1, a_2, a_3\in \RR^3$ so that  $$\int_{S^2} |\zeta - a_i|^2 - |\zeta + a_i|^2 d\mu(\zeta) = 0$$ for $i=1,2,3$.
\end{enumerate}
\end{lemma}

\begin{proof}
Let $a\in \RR^3$, and write $a=(x, y, z)$. Write $\zeta = (\zeta_x, \zeta_y, \zeta_z)\in S^2$, so that
\[
|\zeta - a|^2 - |\zeta + a|^2 = -4 x \zeta_x -4 y \zeta_y - 4 z \zeta_z\ .
\] Integrating against $\mu$ yields
\begin{align*}
\int_{S^2} |\zeta -a|^2 - |\zeta+a|^2 d\mu(\zeta) &= -4 \left(\int_{S^2} \zeta_x \ d\mu(\zeta)\right) x -4 \left(\int_{S^2}\zeta_y\ d \mu(\zeta)\right) y - 4\left(\int_{S^2} \zeta_z \ d\mu(\zeta)\right) z\\
& = -4\left(\int_{S^2}\zeta\ d\mu(\zeta)\right)\cdot a\ ,
\end{align*} where the $\cdot$ in the last expression is the dot product in $\RR^3$. The measure $\mu$ has conformal barycenter 0 if and only if $\int \zeta d\mu(\zeta) = 0$, which happens if and only if $\left(\int_{S^2} \zeta\ d\mu(\zeta)\right) \cdot a = 0$ for all $a\in \RR^3$. This establishes the equivalence of 1 and 2. To see that 3 is equivalent to 2, notice that for fixed $b\in \RR^3$, the expression $b\cdot a = 0$ vanishes for all $a\in \RR$ if and only if it vanishes on three linearly independent vectors $a_1, a_2, a_3 \in \RR^3$. 
\end{proof}

We can now prove Theorem~\ref{thm:gradientRF}
\begin{proof}[Proof of Theorem~\ref{thm:gradientRF}]
We first compute the gradient at $\xi = 0 \in \BB_\CC$. Let $\vv\in T_0 \BB_\CC$ be a unit tangent vector. It determines a unique point $a_{\vv}\in S^2$, and we can find $\tau_{\vv}\in \SU_2(\CC)$ sending the point $(0,0, 1)\in S^2$ to the corresponding $a_{\vv}$. The path $\{t\vv\ :\ t>0\}$ in $\BB_\CC$ corresponds to the path $\{[\tau_{\vv} \cdot \eta_A]\ : \ A >0\}$ in $\SL_2(\CC) / \SU_2(\CC)$. To compute the directional derivative $\partial_{\vv} R_F(0)$ in the ball model is the same as computing 
\[
\lim_{A\to 0} \frac{R_F([\tau \cdot \eta_A]) - R_F([\textrm{id}])}{A} = \lim_{A\to 0} \frac{R_{F^\tau}([\eta_A]) - R_{F^\tau}([\textrm{id}])}{A} = \frac{1}{4}\int_{\PP^1(\CC)} |\zeta - N| ^2 - |\zeta + N|^2 \widehat{\omega_{f^\tau}}(\zeta)\ ,
\] where the last equality is from Proposition~\ref{prop:dirderivRF}. The change of variables by $\tau$ gives
\begin{align*}
\int_{S^2} |\zeta - N| ^2 - |\zeta + N|^2 \widehat{\omega_{f^\tau}}(\zeta)& = \int_{S^2} |\tau^{-1}(\zeta) - N| ^2 - |\tau^{-1}(\zeta) + N|^2\ \widehat{ \omega_f}(\zeta) \\
& = \int_{S^2} |\zeta - \tau(N)| ^2 - |\zeta + \tau(N)|^2\ \widehat{\omega_f}(\zeta)\ ,
\end{align*} where in the last step we have used the invariance of the Euclidean distance on $\RR^3$ under rotations. Recalling that $\tau(N) = a_{\vv}$ and applying Lemma~\ref{lem:barycentercriteria}, we find that
\[
\partial_{\vv} R_F(0) = \frac{1}{4} \int_{S^2} |\zeta - a_{\vv}|^2 - |\zeta + a_{\vv}|^2 \widehat{\omega_f}(\zeta) = -\left(\int_{S^2} \zeta\ \widehat{\omega_f}(\zeta)\right) \cdot a_{\vv}\ .
\] This is to say that
\[
\nabla_h R_F(0) = -\int_{S^2} \zeta\ \widehat{\omega_f}(\zeta)\ .
\]

To compute the gradient at an arbitrary point, we use the transformation formula for the $R_F$: since $R_F([\gamma]) = R_{F^\gamma}([\textrm{id}])$, it follows from the above calculations that $\nabla_h R_F(\xi) = \nabla_h R_{F^\gamma}(0) = -\int_{S^2} \zeta\ \widehat{\omega_{f^\gamma}}(\zeta)$, where $[\gamma]\in \SL_2(\CC) / \SU_2(\CC)$ corresponds to the point $\xi\in \BB_\CC$. 
\end{proof}

The proof of Theorem~\ref{thm:minimizers} now follows:
\begin{proof}[Proof of Theorem~\ref{thm:minimizers}]
By the equivariance of $R_F$, it is enough to assume that $\gamma = \mathbb{1}$. Since $\gamma$ is a minimizer of $R_F$, the gradient of $R_F$ vanishes; by Theorem~\ref{thm:gradientRF} this means $-\int_{S^2} \zeta\ \widehat{\omega_{f}} = 0$, i.e. $\widehat{\omega_f}$ is barycentered.

\end{proof}

We remark that, in the non-Archimedean setting, the analogue of $\omega_f$ is the measure $\delta_f:= f^* \delta_{\zetaG} + f_* \delta_{\zetaG}$, where $\zetaG$ is the Gauss point in the Berkovich projective line over a complete, algebraically closed non-Archimedean field. There is a notion of barycenter in the non-Archimedean context due to Rivera-Letelier, and one can show that $\Bary(\delta_f)$ is contained in the set $\MinResLoc(f)$. In particular, when $\MinResLoc(f)$ is a single point (this always happens when $d$ is even, see \cite{Ru1} Theorem 1.1), then $\MinResLoc(f) = \Bary(\delta_f)$. We expect in the complex setting that $\Min(f)$ is always a single point, and moreover that $\Min(f) = \Bary(\omega_f)$. 

\subsection{$R_F$ and the Projective Capacity}\label{sect:pcap}
In this section, we give a description of the function $R_F$ in terms of Alexandrov's projective capacity. Let $K\subseteq \CC^2$ be compact, and for a function $h:K\to \CC$ let $||h||_K = \sup_{z\in K} |h(z)|$ be the sup norm on $K$. Define
\[
m_k(K) = \inf\left\{ ||Q||_K\ : \ Q\in \CC[X,Y] \textrm{ homogeneous of degree } k, \int_{S^3} \log |Q| d\sigma = k \int_{S^3} \log |z_1| d\sigma\right\}\ .
\] The normalization condition on $Q$ is a multi-dimensional analgue of saying that $Q$ is monic. An explicit calculation shows that $\int_{S^3} \log |z_1| d\sigma = \frac{1}{2}$. The projective capacity is defined to be
\[
\pcap(K) := \lim_{k\to \infty} m_k(K)^{1/k}\ ;
\] the fact that this limit always exists is explained in \cite{Al}. The following Theorem, whose proof will occupy the remainder of this section, shows that $R(F)$ measures the distortion of $\pcap(S^3)$ induced by $F$:
\begin{thm}\label{thm:pcap}
Let $f\in \CC(z)$ and let $F$ be any homogeneous lift. Then
\[
R(F) = d\log \left(\frac{\pcap(S^3)}{\pcap(F^{-1}(S^3))}\right)\ .
\]
\end{thm}

To prove this theorem, we will utilize several extremal functions; the first two are due to Siciak \cite{Sic1}, and the third is the usual pluricomplex Green's function:
\begin{align*}
\Psi_K(z) &:= \sup\left\{ |Q(z)|^{\frac{1}{\textrm{deg } Q}} \ : \ Q\in \CC[X,Y] \textrm{ homogeneous}, ||Q||_K \leq 1\right\}\\
\Phi_K(z) &:= \sup\left\{ |Q(z)|^{\frac{1}{\textrm{deg } Q}} \ : \ Q\in \CC[X,Y], ||Q||_K \leq 1\right\}\\
V_K(z) &:= \sup\left\{ u(z)\ : \ u \textrm{ pluri-subharmonic}, u - \log||z|| = O(1) \textrm{ as } ||z||\to \infty\ , \textrm{ and } u|_K \leq 0\right\}\ .
\end{align*} Note that $\Psi_K(\lambda z) = |\lambda| \Psi_K(z)$ for all $\lambda\in \CC$. As an example, when $K=S^3$, we find $\Psi_{S^3}(z) = ||z||$. 

It is not surprising that there are a number of relationships between these functions; the following are well-known:
\begin{prop}\label{prop:transformationformulas}
Let $K\subseteq \CC^2$ be compact, and let $\Psi_K, \Phi_K, V_K$ be as above. Then
\begin{enumerate}
\item $V_K(z) = \log \Phi_K(z)$
\item If $K$ is circled -- i.e. $w\in K \iff \ e^{i\theta}w \in K \ \forall \theta\in \RR$ -- then $\Phi_K(z) = \max\{1, \Psi_K(z)\}$. 
\item For any homogeneous polynomial $F\in \CC[X,Y]$ of degree $d$, the pluricomplex Green's function satisfies $V_K(F(z)) = d V_{F^{-1}(K)}(z)$.
\end{enumerate}
\end{prop}

\begin{proof}
\begin{enumerate}
\item See Klimek's book \cite{Kli}, Theorem 5.1.7.
\item See \cite{Sic1} Section 9, Theorem 3.
\item See Klimek's book \cite{Kli}, Theorem 5.3.1. 
\end{enumerate}
\end{proof}

Combining these yields the following transformation formula for $\Psi_K$:
\begin{cor}\label{cor:transformpsi}
If $K\subseteq \CC^2$ is compact and circled, and if $F\in \CC[X,Y]$, then 
\[
\Psi_K(F(z)) = \Psi_{F^{-1}(K)}(z)^d\ .
\]
\end{cor}
\begin{proof}
Note that the homgeneity of $F$ implies that $K$ is circled if and only if $F^{-1}(K)$ is circled. 

Combining (1) and (2) of Proposition~\ref{prop:transformationformulas} gives $V_K(z) = \log \max\{1, \Psi_K(z)\}$; since $V_K(z) = \log ||z|| + O(1)$ as $||z|| \to \infty$, we can choose $M$ so that $V_K(z) > 0$ when $||z||\geq M$. Thus, $V_K(z) = \log \Psi_K(z)$ for all $||z|| \geq M$. 

For any $z\in \CC^2$, choose $\lambda$ so that $|\lambda| \cdot ||F(z)|| > M$ and $|\lambda|^{1/d}\cdot ||z|| > M$. Then
\begin{align*}
\log \Psi_K(F(z)) &= \log\left(\frac{1}{|\lambda|} \Psi_K(\lambda \cdot F(z))\right)\\
& = -\log |\lambda| +  V_K(\lambda\cdot F(z))\\
& = -\log|\lambda| + V_K(F(\lambda^{1/d} z)) = (*)\ .
\end{align*} Applying the transformation formula for $V_K$ given in Proposition~\ref{prop:transformationformulas} (3), we see
\begin{align*}
(*) & = -\log |\lambda| + d V_{F^{-1}(K)} (\lambda^{1/d} z)\\
& = -\log |\lambda| + d \log \Psi_{F^{-1}(K)} (\lambda^{1/d} z)\\
& = \log \Psi_{F^{-1}(K)}(z)^d\ .
\end{align*}
\end{proof}

Finally, we will use one last theorem relating the projective capacity to the function $\Psi_K(z)$:
\begin{prop}\label{prop:Cegrellpcap} (See \cite{Ceg})
Let $K\subseteq \CC^2$ be compact. Then
\[
\log \pcap(K) = -\frac{1}{2} - \int_{S^3} \log \Psi_K d\sigma\ .
\]
\end{prop}

We are now ready to prove Theorem~\ref{thm:pcap}
\begin{proof}[Proof of Theorem~\ref{thm:pcap}]
Observe that
\[
R(F) = \int_{S^3} \log ||F(z)|| d\sigma = \int_{S^3} \log \Psi_{S^3}(F(z)) d\sigma\ .
\] Applying the transformation formula in Corollary~\ref{cor:transformpsi} and the identity in Proposition~\ref{prop:Cegrellpcap} yields
\begin{equation}\label{eq:almostsimplifiedRFcapacity}
R(F) = d\int_{S^3} \log \Psi_{F^{-1}(S^3)}(z) d\sigma = -\frac{d}{2} - d\log \pcap(F^{-1}(S^3))\ .
\end{equation} But from Proposition~\ref{prop:Cegrellpcap} and the fact that $\Psi_{S^3}(z) = ||z||$, we see that $\log \pcap(S^3) = -\frac{1}{2}$. Plugging this into (\ref{eq:almostsimplifiedRFcapacity}) gives
\[
R(F) = d \log \pcap(S^3) - d\log \pcap(F^{-1}(S^3))
\] which is the equality asserted in the theorem.
\end{proof}
\section{Asymptotic behavior}\label{sect:asymp}
\subsection{Convergence of $\frac{1}{d^n} R_{F^n}$}

Our goal in this section is to compute the asymptotic limit of $d^{-n} R_{F^{(n)}}$ as a function on hyperbolic space, and to give a geometric interpretation to the limiting function. We will show that
\begin{thm}\label{thm:RFconvergence}
The functions $\frac{1}{d^n} R_{F^{(n)}}$, viewed as functions on the hyperbolic ball $\BB_{\CC}$, converge locally uniformly to $h_{\mu_f} + C_F$, where $\mu_f$ is the measure of maximal entropy for $f$ and $h_{\mu}$ is averaged Busemann function introduced in Section~\ref{sect:barycentersDE}. 
\end{thm} As in previous sections, while this statement is phrased in terms of the ball model $\BB_\CC$, we will carry out different computations in whichever model is particularly convenient, and then transfer these back to the ball model.

Recall also that the function $h_{\mu_f}$ is minimized precisely on the conformal barycenter of $\mu_f$. We will show in the following section that the sets $\Min(f)$ converge in the Hausdorff topology to $\Bary(\mu_f)$.

Given $\gamma\in \SL_2(\CC)$, write $\gamma = \tau\cdot \eta_A \cdot \sigma$ for some $\tau, \sigma \in \SU_2(K)$ and $A>0$. By the $\SU_2(\CC)$-invariance of the norm on $\CC^2$, we find that
\[
||F^\gamma(X,Y)|| = || \sigma^{-1} \eta_A ^{-1} \tau^{-1} F(\gamma\cdot (X,Y))|| = ||\eta_A^{-1} \cdot \tau^{-1} F(\gamma\cdot(X,Y))||\ .
\]Applying Lemma~\ref{lem:precomp} we find that
\[
e^{-A/2} ||F(\gamma\cdot(X,Y))|| \leq ||F^\gamma(X,Y)|| \leq e^{A/2} ||F(\gamma\cdot(X,Y))||\ .
\] Applying this to the iterates $F^{(n)}$ of $F$ yields
\[
e^{-A/2} ||F^{(n)}(\gamma\cdot(X,Y))|| \leq ||(F^{(n)})^\gamma(X,Y) || \leq e^{A/2} ||F^{(n)}(\gamma\cdot(X,Y))||
\] and so
\begin{equation}\label{eq:preconvergenceestimate}
-\frac{A}{2d^n}+  d^{-n} \int_{S^3} \log ||F^{(n)}(\gamma\cdot(X,Y))|| d\Vol_{S^3}\leq d^{-n} R_{F^n}([\gamma]) \leq \frac{A}{2d^n} + d^{-n} \int_{S^3} \log||F^{(n)}(\gamma\cdot(X,Y))|| d\Vol_{S^3}\ .
\end{equation}

\begin{lemma}\label{lem:firstconvergence}
As $n\to \infty$, $d^{-n} R_{F^{(n)}}([\gamma])$ converges locally uniformly on $\SL_2(\CC) / \SU_2(\CC)$ to
\[
\widetilde{H}_F([\gamma]) := \int_{S^3} H_F(\gamma\cdot(X,Y)) d\Vol_{S^3}\ ,
\] where $H_F(\cdot)$ is the homogeneous escape rate function of $F$.
\end{lemma}
\begin{proof}
We first establish the convergence result. This, in essence, follows directly from the definition of $H_F$: recall that
\[
H_F(X,Y) = \lim_{n\to\infty} d^{-n} \log ||F^{(n)}(X,Y)||\ ,
\] and this convergence is locally uniform on $\CC^2$. The uniformity allows us to pass the limit into the integrals in (\ref{eq:preconvergenceestimate}), and we find that
\[
d^{-n} R_{F^{(n)}}([\gamma]) \to \int_{S^3} H_F(\gamma\cdot (X,Y)) d\Vol_{S^3}\ .
\]The fact that this is locally uniform on $\SL_2(\CC) / \SU_2(\CC)$ comes from (\ref{eq:preconvergenceestimate}), where we see that the error term depends only on $A = \dSL([\gamma], [\textrm{id}])$ (there is also dependence on $F$, coming from the limit defining $H_F$). 
\end{proof}

The escape rate function $H_F$ on $\CC^2$ descends to give the Green's function $g_F([x:y]) := H_F(x,y) - \log ||x,y||$ on $\PP^1(\CC)$. The limit function in the preceeding Lemma can be expressed in terms of the hyperbolic harmonic extension $\H\{g_F\}$. Write
\begin{align}
\int_{S^3} H_F(\gamma\cdot (X,Y)) d\Vol_{S^3} &= \int_{S^3} H_F(\gamma\cdot(X,Y)) - \log||\gamma\cdot (X,Y)|| d\Vol_{S^3} + \int_{S^3} \log ||\gamma\cdot (X,Y)|| d\Vol_{S^3}\label{eq:firstexpandlimit}\ .
\end{align} Rewriting the first integral in affine coordinates gives
\begin{align*}
\int_{S^3} H_F(\gamma\cdot (X,Y)) - \log ||\gamma\cdot (X,Y)|| d\Vol_{S^3} &= \int_{\PP^1(\CC)} g_F(\gamma(\alpha))\  \omega_{\CC}\\
& = \H\{g_F\}(\gamma(j))\ ,
\end{align*} where we recall that $\gamma(j) \in \hh_{\CC}$ is the point corresponding to $[\gamma]$, and in the last equality we have used the formula for $\H\{g_F\}$ given in Proposition~\ref{prop:hyperbolicextensions}. Returning to (\ref{eq:firstexpandlimit}), we now consider the function
\[
\mathcal{I}([\gamma]) = \int_{S^3} \log ||\gamma\cdot(X,Y)|| d\Vol_{S^3}\ .
\] 
Note that this is a radial function, in that it only depends on $\dSL([\gamma], [\textrm{id}])$: writing $\gamma = \tau \cdot \eta_A \cdot \sigma$ as above, we find that
\[
\int_{S^3} \log ||\gamma \cdot (X,Y)|| d\Vol_{S^3} = \int_{S^3} \log || \eta_A (X,Y)|| d\Vol_{S^3}\ .
\] It will be most convenient to treat this as a function on $(\BB_\CC, \dB)$. 
\begin{lemma}
Identifying $\SL_2(\CC) / \SU_2(\CC)$ with $\BB_{\CC}$, so that $\mathcal{I}: \BB_{\CC} \to \RR$, we find
\[
\Delta_h^{\BB_\CC} \mathcal{I} = 4\ .
\]
\end{lemma}
\begin{proof}
Write $\mathcal{I}([\gamma]) = \int_{S^3} \log ||\eta_A \cdot (X,Y)|| d\Vol_{S^3}$ as was noted above. In spherical coordinates on $\BB_\CC$, we recall that the hyperbolic Laplacian takes the form
\begin{equation}\label{eq:hypLapballradial}
\Delta_h^{\BB_\CC} = \frac{1-r^2}{r^2}\left( (1-r^2) N^2 +(1+r^2) N + (1-r^2) \Delta_\sigma\right)\ ,
\end{equation} where $N = r\partial_r$ and $\Delta_\sigma$ is the angular part of the Euclidean Laplacian. Since $\mathcal{I}$ is radial, $\Delta_\sigma \mathcal{I} = 0$.

To compute the derivative in the radial direction, we first evaluate the integral defining $\mathcal{I}([\gamma])$ explicitly: note that 
\begin{align*}
\mathcal{I}([\gamma]) &= \int_{S^3} \log ||\eta_A \cdot (X,Y)||d\Vol_{S^3}\\
& = \frac{1}{2} \int_{\PP^1(\CC)} \log\frac{e^{2A} |\alpha|^2 + 1}{|\alpha|^2 + 1}\ \omega_{\CC} - \frac{A}{2}\\
& = \frac{1}{2} \int_{\PP^1(\CC)} \log \left(e^{2A} |\alpha|^2 + 1\right) \omega_{\CC} - \frac{A+1}{2}\ .
\end{align*} Expressing $\omega_{\CC}$ in polar coordinates yields
\begin{align*}
\mathcal{I}([\gamma]) & = \frac{1}{2} \int_0^\infty \frac{2r \log\left(e^{2A} r^2 + 1\right)}{1+r^2} dr - \frac{A+1}{2}\ .
\end{align*} The integral can be computed explicitly, giving
\begin{equation}\label{eq:Isimplified}
\mathcal{I}([\gamma]) = \frac{A e^{2A}}{e^{2A}-1}  - \frac{A+1}{2} = \frac{1}{2} \left( A \left( \frac{e^{2A} +1}{e^{2A} - 1}\right) - 1\right)\ .
\end{equation} Identifying $[\gamma]$ with its image $\xi\in \BB_{\CC}$, we find that 
\[
A = \dSL([\gamma], [\textrm{id}]) = \dB(\xi, 0) = \log \left(\frac{1+|\xi|}{1-|\xi|}\right)\ .
\] Inserting this into (\ref{eq:Isimplified}) and simplifying gives
\begin{align}
\mathcal{I}(\xi) &= \frac{1}{2} \left(\log \left( \frac{1+|\xi|}{1-|\xi|}\right)\cdot \left(\frac{1+|\xi|^2}{2|\xi|}\right) - 1\right)\nonumber\\
& = \frac{1}{2} \left(\log \left( \frac{1+r}{1-r}\right)\cdot\left(\frac{1+r^2}{2r}\right) - 1\right)\label{eq:explicitmathcalIfunction}\ .
\end{align} Finally, putting this into the expression for the hyperbolic Laplacian $\Delta_h^{\BB_\CC}$ given in (\ref{eq:hypLapballradial}) and simplifying gives
\[
\Delta_h^{\BB_\CC} \left(\mathcal{I}(\cdot)\right) = 4\ ,
\] from which the assertion in the lemma follows. 
\end{proof}

To summarize what has been shown so far, by identifying a class $[\gamma]\in \SL_2(\CC) / \SU_2(\CC)$ with the corresponding point $\xi \in \BB_{\CC}$, we see that $d^{-n} R_{F^{(n)}}$ converges locally uniformly to a function
\[
\Gamma_F(\xi):= \H\{g_F\}(\xi) + \mathcal{I}(\xi)
\] on $\BB_\CC$ with $\Delta_h^{\BB_\CC}\ \Gamma_F = 4$.

In Proposition~\ref{prop:hyplapDE} we showed that for any probability measure $\mu$ on $S^2$, the function $h_\mu:\BB_{\CC} \to \RR$ given $h_\mu(z):= -\frac{1}{2}\int_{S^2} \log \left(\frac{1-|z|^2}{|z-\zeta|^2}\right) d\mu(\zeta)$ satisfies
\[
\Delta_h^{\BB_\CC} h_\mu = 4\ .
\] It follows that
\[
\Gamma_F - h_\mu
\] is a hyperbolic harmonic function for any probability measure $\mu$ on $S^2$. We will be particularly interested in the case that $\mu = \mu_f$ is the (pullback of the) measure of maximal entropy of $f$.

\begin{prop}
The function $\Gamma_F-  h_{\mu_f}$ is constant on $\BB_{\CC}$. More precisely,
\[
\Gamma_F(\xi) = h_{\mu_f}(\xi) + \int_{\PP^1(\CC)} g_F \omega_\CC - 1\ .
\]
\end{prop}
\begin{proof}
To show that this function is constant, we will consider the boundary behavior of $\Gamma_F - h_{\mu_f}$. Write
\[
\Gamma_F(\xi) = \H\{g_F\}(\xi) + \mathcal{I}(\xi)
\] for $\xi\in \BB_{\CC}$, and 
\[
h_{\mu_f}(\xi) = -\frac{1}{2} \int_{S^2} \log \left(\frac{1- |\xi|^2}{|\zeta-\xi|^2}\right) d\mu_f(\zeta)\ .
\] The difference $\Gamma_F(\xi) -  h_{\mu_f}$ can be written
\[
\Gamma_F(\xi) -  h_{\mu_f}(\xi) =  \left(\H\{g_F\}(\xi)-\int_{S^2} \log |\zeta-\xi|\ d{\mu_f}(\zeta) \right) + \left(\mathcal{I}(\xi) + \frac{1}{2} \log(1-|\xi|^2)\right)\ .
\] 

We consider separately the two terms appearing here:
\begin{itemize}
\item By Proposition~\ref{prop:logconvergencesphere}, since $\mu_f$ has continuous potentials (see, e.g. \cite{DS} Th\'eor\`eme 3.7.1) we find that  
\[
\lim_{\xi \to \xi_0\in S^2} \int_{S^2} \log |\zeta- \xi| d\mu_f(\zeta) = \int_{S^2} \log |\zeta-\xi_0| d\mu_f(\zeta)
\]
Noticing that $|\zeta - \xi_0| = 2 ||\zeta, \xi_0||$ (where we are writing $\widetilde{\zeta}, \widetilde{\xi_0}$ for the points in $\PP^1(\CC)$ corresponding to $\zeta, \xi_0\in S^2$ under stereographic projection), the limiting function in the above expression can also be given
\begin{align*}
\int_{S^2}\log |\zeta - \xi_0| d\mu_f (\zeta) &= \int_{\PP^1(\CC)} \log ||\widetilde{\zeta}, \widetilde{\xi_0}|| d\mu_f(\zeta) + \log 2
\end{align*} The function $g_F([x:y]) = H_F(x,y) - \log ||x,y||$ introduced above satisfies $dd^c g_F = \mu_{f} - \omega_\CC$ on $\PP^1(\CC)$ (see \cite{BaBe} Section 1.3), so that
\begin{align}
\int_{S^2} \log |\zeta - \xi_0| d\mu_f &= \int_{\PP^1(\CC)} \log ||\widetilde{\zeta},\widetilde{\xi_0}|| dd^c g_F + \int_{\PP^1(\CC)} \log ||\widetilde{\zeta}, \widetilde{\xi_0}|| \omega_\CC + \log 2\nonumber\\
& = \int_{\PP^1(\CC)} g_F (\delta_{\widetilde{\xi_0}}-\omega_\CC) + \log 2 = g_F(\widetilde{\xi_0}) - \int_{\PP^1(\CC)} g_F \omega_\CC + \log 2\ ,\label{eq:limitofintegralonsphere}
\end{align} where we are using that $\int_{\PP^1(\CC)} \log ||a, b|| \omega_\CC(a) = 0$ for any fixed $b\in \PP^1(\CC)$ (this can be checked directly with an explicit calculation in local coordinates). Let $\kappa_1(F): = \int_{\PP^1(\CC)} g_F\omega_\CC$. 

Since $\H\{{g_F}\}$ is the hyperbolic harmonic extension of $g_F$, it follows that, as $\xi \to \xi_0\in S^2$, $\H\{{g_F}\}(\widetilde{\xi}) \to {g_F}(\widetilde{\xi_0})$; combining this with (\ref{eq:limitofintegralonsphere}) gives that
\begin{equation}\label{eq:firstpartofboundarybehavior}
\lim_{\BB_\CC \ni \xi \to \xi_0 \in S^2} \left( \H\{g_F\}(\xi) - \int_{S^2} \log | \zeta - \xi| d\mu_f(\zeta) \right) = \kappa_1(F) - \log 2\ .
\end{equation}

\item We next consider the expression $\mathcal{I}(\xi) - \frac{1}{2} \log(1-|\xi|^2)$. Using the formula for $\mathcal{I}(\xi)$ derived in (\ref{eq:explicitmathcalIfunction}), this quantity is
\[
\mathcal{I}(\xi) - \frac{1}{2} \log(1-|\xi|^2) = \frac{1}{2} \left(\log \left( \frac{1+r}{1-r}\right)\cdot\left(\frac{1+r^2}{2r}\right) - 1\right) + \frac{1}{2} \log(1-r^2)\ ,
\] where $r=|\xi|$. Simplifying this expression gives
\[
\mathcal{I}(\xi) - \frac{1}{2} \log(1-|\xi|^2) = \frac{1}{2} \log (1+r) \left(\frac{(1+r)^2}{2r}\right) - \frac{1}{2} \log (1-r) \left(\frac{(r-1)^2}{2r}\right) - 1\ .
\] The limit as $r\to 1$ can be computed directly using L'H\^opital's rule, and we find
\begin{equation}\label{eq:secondpartofboundarybehavior}
\lim_{\BB_{\CC}\ni \xi \to \xi_0\in S^2} \left(\mathcal{I}(\xi) - \frac{1}{2} \log (1-|\xi|^2)\right)  = (\log 2) - 1\ .
\end{equation}
\end{itemize}

Combining (\ref{eq:firstpartofboundarybehavior}) and (\ref{eq:secondpartofboundarybehavior}) we find that 
\[
\lim_{\BB_\CC \ni \xi \to \xi_0 \in S^2} \Gamma_F(\xi) - h_{\mu_f}(\xi) = \kappa_1(F) - 1 = \int_{\PP^1(\CC)} g_F \omega_C - 1\ .
\]

Since $\Gamma_F - h_{\mu_f}$ is a harmonic function that is constant at the boundary of $\BB_\CC$, it must be constant throughout $\BB_\CC$, equal to its boundary value. Thus, $\Gamma_F - h_{\mu_f} = \int_{\PP^1(\CC)} g_F \omega_\CC - 1$ as asserted.
\end{proof}

Since $h_{\mu_f}$ is minimized on the conformal barycenter of $\mu_f$ (viewed as a measure on $S^2$ via pullback under stereographic projection), we find that
\begin{cor}\label{cor:minimizerisbary}
The limiting function $\Gamma_F$ is minimized on the conformal barycenter of the measure of maximal entropy $\mu_f$.
\end{cor}

\subsection{Convergence of the Minimizing Sets}\label{sect:convergenceofmins}
In this section, we prove the following theorem concerning the sets $\Min(f^n)$:
\begin{thm}\label{thm:Minfconvergence}
The sets $\Min(f^n)$ converge in the Hausdorff topology to the barycenter $\Bary(\mu_f)$ of the measure of maximal entropy $\mu_f$.
\end{thm}

This theorem should not be surprising. The functions $R_{F^n}$ converge locally uniformly to $h_{\mu_f}+ C_F$ for an explicit constant $C_F$ (see Theorem~\ref{thm:RFconvergence}), the sets $\Min(f^n)$ are the minimizers for $R_{F^n}$, and $\Bary(\mu_f)$ is the minimizer of $h_{\mu_f}$. 

However, it's important to note that uniform convergence is not, in general, enough to ensure the Hausdorff convergence of minimizing sets. The extra ingredient needed in our context is given by the following result:

\begin{prop}\label{prop:minsetsarebounded}
There exists $R>0$ so that the sets $\Min(f^n) \subseteq B_R(0)$ for all $n\in \NN$.
\end{prop}
\begin{proof}
Transferring Theorem~\ref{thm:growthrates} to the ball model and applying it to the iterate $F^n$ yields
\begin{equation}\label{eq:growthestimatesballmodel}
\frac{d^n-1}{2} \dB(\zeta, 0) + \log C_1(F^n) \leq R_{F^n}(\zeta) \leq \frac{d^n+1}{2} \dB(\zeta, 0) + \log C_2(F^n)
\end{equation} Note that
\[
C_1(F) ||F^{n-1}(X,Y)||^d \leq ||F^n(X,Y)|| = ||F(F^{n-1}(X,Y))|| \leq C_2(F) ||F^{n-1}(X,Y)||^d
\] for any $||X,Y|| = 1$. Applying this inductively, we find
\[
C_1(F)^{\frac{d^n-1}{d-1}} \leq ||F^n(X,Y)|| \leq C_2(F)^{\frac{d^n-1}{d-1}}
\] for all $n$ and all $||X,Y|| =1$; in particular,
\[
C_1(F)^{\frac{d^n-1}{d-1}} \leq C_1(F^n) \ \textrm{ and } \ C_2(F^n) \leq C_2(F)^{\frac{d^n-1}{d-1}}\ .
\]Inserting this into (\ref{eq:growthestimatesballmodel}) yields
\begin{equation}\label{eq:refinedgrowthestimatesball}
\frac{d^n-1}{2} \dB(\zeta, 0) + \frac{d^n-1}{d-1} \log C_1(F) \leq R_{F^n}(\zeta) \leq \frac{d^n+1}{2} \dB(\zeta, 0) + \frac{d^n-1}{d-1} \log C_2(F)\ .
\end{equation} Thus, the largest that $R_{F^n}(0)$ can be is $\frac{d^n-1}{d-1} \log C_2(F)$. But note that, if
\[
\dB(\zeta, 0) > \frac{2}{d-1} \log \frac{C_2(F)}{C_1(F)} :=R
\] then (\ref{eq:refinedgrowthestimatesball}) implies
\[
R_{F^n}(0) \leq \frac{d^n-1}{d-1} \log C_2(F) < \frac{d^n-1}{2} \dB(\zeta, 0) + \frac{d^n-1}{d-1} \log C_1(F) \leq R_{F^n}(\zeta)\ .
\] In particular, $R_{F^n}(\zeta)$ is not a minimizer of $R_{F^n}$. Hence, $R_{F^n}$ must be minimized on $B_R(0)$, i.e $\Min(f^n) \subseteq B_R(0)$. Since $R$ is independent of $n$, we see that $\Min(f^n)\subseteq B_R(0)$ for all $n\in \NN$ as claimed.
\end{proof}

We also make use of the following basic fact from analysis
\begin{lemma}\label{lem:analysislemma}
Suppose $f_n$ is a sequence of functions on a metric space $(X,d)$ that converges locally uniformly to a continuous function $f$ on $X$. If $x_n\in X$ is a  sequence of minimizers of $f_n$, i.e. for any $y\in X$ we find $f_n(y) \geq f_n(x_n)$, and if $x_n \to \hat{x}\in X$, then $\hat{x}$ is a minimizer of $f$. 
\end{lemma}

We are ready to prove Theorem~\ref{thm:Minfconvergence}:
\begin{proof}[Proof of Theorem~\ref{thm:Minfconvergence}]
Since $\Bary(\mu_f)=\{y\}$ is a single point, it's enough to show that for every $\epsilon>0$, can choose $N$ so that $\Min(f^n)\subseteq B_\epsilon(y)$ for all $n\geq N$. Suppose this is not the case; then there exists $\epsilon_0 >0$ and a sequence of points $x_{n_k}\in \Min(f^{n_k})$ with $x_{n_k}\not\in B_{\epsilon_0}(y)$ for all $k\in \NN$. By Proposition~\ref{prop:minsetsarebounded}, the $x_{n_k}$ all lie in some fixed $B_R(0)$, hence we can extract a convergent subsequence again written $x_{n_k} \to \hat{x}$. By Lemma~\ref{lem:analysislemma}, the limit $\hat{x}$ is a minimizer of the limiting function $\Gamma_F(\cdot) = h_{\mu_f}(\cdot) + C_F$ of the functions $\frac{1}{d^n} R_{F^n}$; in particular, $\hat{x}\in \Bary(\mu_f) = \{y\}$, i.e. $\hat{x} = y$. But this is a contradiction, since $\dB(x_{n_k}, y) > \epsilon_0$ for all $k$. 
\end{proof}

\section{Case Study: $d=1$}\label{sect:done}
In this section we explicitly compute the set $\Min(f)$ and the min-invariant for rational maps $f\in K(z)$ of degree 1, where $K$ can be $\RR$ or $\CC$. Up to conjugacy, any map $f\in K(z)$ with degree $d=1$ can be written in one of the following three forms; note that the lifts have been chosen so that $F(X,Y) = M\cdot(X,Y)^\top$ for the appropriate $M\in \SL_2(K)$:
\begin{enumerate}
\item $f(w) = w$, with lift $F(X,Y) = (X,Y)$;
\item $f(w) = w+1$ with lift $F(X,Y) = (X+Y, Y)$, or
\item $f(w) = \lambda w$ for $\lambda\in K\setminus\{0,1\}$, with $F(X,Y) = (\lambda^{1/2} X , \lambda^{-1/2} Y)$.
\end{enumerate}

In each case, $F(X,Y) = A\cdot (X,Y)^\top$ for some matrix in $A\in \SL_2(K)$; a direct calculation shows that $|\Res(F)| = |\det(A)|^2 = 1$. Thus, to compute $m_K(f)$ it suffices to determine the minimum value of $R_F$ for the particular lifts given above.

\begin{lemma}\label{lem:minilemmadist}
Let $f(w) = aw+b$ and let $\gamma(w) = tw+z$. Then
\[
\cosh(\dH(f^\gamma(j), j)) = 1+ \frac{\left|\frac{f(z)-z}{t}\right|^2 + (|a|-1)^2}{2|a|}\ .
\]
\end{lemma}
\begin{proof}
This is a straightforward calculation. Let $f(w) = aw + b$, and $\gamma(w) =  tw + z$. Then $f^{\gamma}(w) = aw + \frac{f(z)-z}{t}$, from which we find that
\begin{align*}
f^\gamma(j) =  |a| j + \frac{f(z)-z}{t}\ .
\end{align*} By (\ref{eq:hypdist}), we have
\begin{equation}\label{eq:hypdistconjugate}
\cosh(\dH(f^\gamma(j), j)) = 1+ \frac{\left|\frac{f(z)-z}{t}\right|^2 + (|a| - 1)^2}{2|a|}\ .
\end{equation} 

\end{proof}

Let $f(w) = aw+b$, and let $F(X,Y) = \left(\begin{matrix} \sqrt{a} & \frac{b}{\sqrt{a}}\\ 0 & \frac{1}{\sqrt{a}}\end{matrix}\right) \cdot (X,Y)^\top$ be the homogeneous lift normalized to have determinant 1. The conjugate $f^\gamma$ can be lifted to a map $F^\gamma(X,Y) = M_\gamma\cdot (X,Y)^\top$; recall that $\dH(f^\gamma(j), j)$ computed in Lemma~\ref{lem:minilemmadist} is precisely the exponent $A$ appearing in the decomposition
\[
M_{\gamma} = \tau \cdot \left(\begin{matrix} e^{A/2} & 0 \\ 0 & e^{-A/2} \end{matrix}\right) \cdot \sigma
\] where $\tau, \sigma\in \SU_2(K)$. In Lemma~\ref{lem:simplestcalculationRF} we saw that
\[
R(M_\gamma) = \int_{S^3} \log ||M_\gamma \cdot (X,Y)^\top|| d\Vol_{S^3} =  \left\{\begin{matrix} \log(1+e^A) - \frac{A}{2} - \log 2\ , & K = \RR\\[5pt] -\frac{1}{2} + \frac{A}{2} \left(\frac{e^{2A} +1}{e^{2A}-1}\right)\ , & K = \CC \end{matrix}\right.\ .
\]

The following lemma says that $R(M_\gamma)$ is increasing as a function of $A = \dH(f^\gamma(j), j)$:

\begin{lemma}
The functions $x\mapsto \log(1+e^x) - \frac{x}{2}$ and $x\mapsto x\coth(x)$ are increasing on $(0, \infty)$.
\end{lemma}
\begin{proof}
For the first function, we compute
\[
\partial_x \left(\log(1+e^x) - \frac{x}{2}\right) = \frac{e^x - 1}{e^x+1}\ ,
\] which is strictly positive for $x\in (0,\infty)$. For the second function, note $x\coth(x) = x\left(\frac{e^{2x}+1}{e^{2x}-1}\right)$, whose derivative in $x$ is
\[
\partial_x \left(x\coth(x)\right) = \frac{e^{4x} - 4x e^{2x} -1}{(e^{2x}-1)^2}\ .
\] The numerator can be expanded as a Taylor series
\[
e^{4x} - 4x e^{2x} -1 = \sum_{n=1}^\infty \left(\frac{4^n}{n!} - \frac{2^{n+1}}{(n-1)!}\right) x^n\ ;
\]note that the general term in this series can be simplified to $\frac{4^n - n 2^{n+1}}{n!}$; for $n=1,2$ this term is 0, but for $n\geq 3$ this term is strictly positive. Consequently, for $x>0$ the series is strictly positive, i.e. $\partial_x (x\coth(x)) >0$ on $(0,\infty)$.
\end{proof}

We are now ready to compute $\Min(f)$ and the min-invariant for maps of degree $d=1$. We follow the cases outlined above:

\begin{enumerate}
\item If $f(\alpha) = \alpha$, then $R_F$ is constant, so that $\Min(f) = \hh_K$ and $m_K(f) = 1$, where $F(X,Y) = (X,Y)$ is the trivial lift of $f$ to $\SL_2(K)$. 

\item If $f(\alpha)$ is conjugate to $\alpha\mapsto\alpha+1$, then by Lemma~\ref{lem:minilemmadist} we find that
\[
\cosh(\dH(f^\gamma(j), j)) = 1 + \frac{1}{2t^2}\ .
\] Note that this is independent of $z$. As $t$ varies in $(0,\infty)$, the fact that $\cosh(x)$ is monotonically increasing implies that $A(t)=\dH(f^\gamma(j), j)$ is decreasing; consequently, $R_F([\gamma_{z,t}])$ is also decreasing as $t$ varies in $(0,\infty)$. An explicit calculation shows that
\begin{align*}
\lim_{t\to 0} R_F([\gamma_{z,t}]) &= \infty\ , \textrm{ and }\\
\lim_{t\to\infty}R_F([\gamma_{z,t}]) & = 0\ . 
\end{align*} Thus, along the oriented geodesics $[\infty,z]$ in $\hh_K$, the function $R_F$ increases monotonically from 0 to $\infty$. We conclude that $R_F$ attains its minimum at $\{\infty\} \in \partial \hh_K$, and that its minimum value is 0.

\item If $f(\alpha)$ is conjugate to $\alpha \mapsto \lambda\alpha$ for $\lambda\in K\setminus \{0, 1\}$, then by Lemma~\ref{lem:minilemmadist} we have
\[
\cosh(\dH(f^\gamma(j), j)) = 1 + \frac{(|\lambda| -1)^2}{2|\lambda|} + \frac{|\lambda-1|^2 \cdot |z|^2}{2|\lambda| t}\ .
\]Note that, if $z=0$, then $\dH(f^\gamma(j), j)$ is constant; hence, $R_F$ is constant along the hyperbolic geodesic $[0, \infty]$, and we can explicitly compute
\[
R_F([\gamma]) = R_F([\textrm{id}]) = \left\{\begin{matrix} \log(1+|\lambda|) - \frac{\log|\lambda|}{2} - \log 2\ , & K = \RR \\[5pt] -\frac{1}{2} + \frac{\log|\lambda|}{2} \cdot \left(\frac{|\lambda|^2 + 1}{|\lambda|^2 - 1}\right)\ , & K = \CC \end{matrix}\right.\ .
\]Otherwise, we argue as in the previous case, noting that as $t$ varies in $(0,\infty)$ we find that $A(t) = \dH(f^\gamma(j), j)$ is decreasing, so that $R_F([\gamma_{z,t}])$ is also decreasing as $t$ varies in $(0,\infty)$. Now an explicit calculation shows that
\begin{align*}
\lim_{t\to 0} R_F([\gamma_{z,t}]) &= \infty\ , \textrm{ and }\\
\lim_{t\to\infty}R_F([\gamma_{z,t}]) & = R_F([\textrm{id}])\ . 
\end{align*} Thus, we've shown that $R_F$ attains its minimum value $m_K(f)$ along the entire geodesic $[0,\infty]$, and off of this geodesic $R_F$ is strictly larger than this value.
\end{enumerate}

\end{document}